\documentclass{article}

\usepackage{amsmath,amsfonts,amssymb,amsthm,graphicx,url,subfigure,float}
\usepackage[colorlinks=true]{hyperref}
\usepackage{pdfsync}
\usepackage{dsfont}

\newcommand{\BE}{\begin{equation}}
\newcommand{\EE}{\end{equation}}
\newtheorem{propo}{{P}roposition}

\newtheorem{defi}{{D}efinition}
\newtheorem{lem}{{L}emma}
\newtheorem{theo}{{T}heorem}
\newtheorem*{theos}{{T}heorem}
\newtheorem*{propos}{{P}roposition}
\newtheorem{cor}{{C}orollary}

\newcommand{\fr}{\frac{1}{2}}
\newcommand{\I}{\mathbb{I}}

\newcommand{\N}{\mathbb{N}}
\newcommand{\R}{\mathbb{R}}

\newcommand{\h}{\mathcal{H}}

\renewcommand{\leq}{\leqslant}
\renewcommand{\geq}{\geqslant}

\theoremstyle{definition}\newtheorem{remark}{Remark}
\theoremstyle{definition}\newtheorem{example}{Example}

\title{Sub-Riemannian Geometry and Geodesics in Banach Manifolds}

\date{}

\author{Sylvain Arguill\`ere}

\begin{document}

\maketitle

\begin{abstract}
In this paper, we define and study sub-Riemannian structures on Banach manifolds. We obtain extensions of the Chow-Rashevski theorem for exact controllability, and give conditions for the existence of a Hamiltonian geodesic flow despite the lack of a Pontryagin Maximum Principle in the infinite dimensional setting.
\end{abstract}

\section*{Introduction}

A sub-Riemannian manifold is a smooth manifold $M$ of finite dimension, endowed with a distribution of subspaces $\Delta\subset TM$, together with a smooth Riemannian metric $g$ on $\Delta$ \cite{SRBOOK,MBOOK}. This allows the definition of horizontal vector fields (resp. horizontal curves) which are those that are almost everywhere tangent to the distribution. We can then define the length and action of a horizontal curve thanks to $g$ and, just like in Riemannian manifolds, the sub-Riemannian distance between two points, and the notion of sub-Riemannian geodesic.

There are two foundational results in sub-Riemannian geometry:
\begin{itemize}
\item The Chow-Rashevksi theorem (generalized by Sussmann to the orbit theorem \cite{SUSS}). It states that if the iterated Lie brackets of horizontal vector fields span the whole tanget bundle, then any two points can be connected by a horizontal curve. This is the problem of \textit{controllability}.
\item The Pontryagin Maximum Principle (PMP) \cite{PBOOK} from optimal control theory. It states that there are two kinds of sub-Riemannian geodesics. 
\begin{itemize}
\item The \textit{normal} geodesics, which are the projection to $M$ of the Hamiltonian flow
$$\left\lbrace
\begin{aligned}
\dot{q}(t)&=\partial_ph(q(t),p(t)),\\
\dot{p}(t)&=-\partial_qh(q(t),p(t)),
\end{aligned}\right.
$$
where the Hamiltonian $h$ of the system is defined on $T^*M$ by
$$
h(q,p)=\max_{v\in\Delta_q} (p(v)-\frac{1}{2}g_q(v,v)),\quad q\in M,\ p\in T_q^*M.
$$
The converse is true: any such projection $q(\cdot)$ of the Hamiltonian flow is indeed a geodesic.
\item The \textit{abnormal} geodesics, which are among the singular curves of the structure. Singular curves only depend on the distribution $\Delta$, and are the projection to $M$ of curves $t\mapsto (q(t),p(t))$ on $T^*M$ that satisfy the so-called \textit{abnormal Hamiltonian equations}
$$\left\lbrace
\begin{aligned}
\dot{q}(t)&=\partial_pH^0(q(t),p(t),v(t)),\\
\dot{p}(t)&=-\partial_qH^0(q(t),p(t),v(t)),\\
0&=\partial_vH^0(q(t),p(t),v(t)),
\end{aligned}\right.
$$
with $v(t)=\dot{q}(t)$ and $H(q,p,v)=p(v),$ $q\in M$, $v\in\Delta_q$, $p\in T^*_qM$. The converse is not true in this case: there are some singular curves that are not geodesics.
\end{itemize} 
\end{itemize}

The purpose of this paper is to lay the foundation to infinite dimensional sub-Riemannian geometry with a wide range of distributions $\Delta$ (for example, $\Delta$ might be dense in infinite dimensions, as in \cite{ATY,AT}), and generalize those two results. Many difficulties appear, since none of the methods used in finite dimensions work. For example, it is well-known that there is no PMP in infinite dimensions \cite{LY}, though some work has been done in this direction in \cite{GMV} for certain special cases. We will also see that the controllability problem (and more generally, the problem of finding the set of points to which a fixed $q$ can be connected) is much harder to solve.

Another problem is the possibility that $g$ be a \textit{weak metric}, that is, that the norm it induces on each horizontal subspace is not complete. This makes the problem of existence of geodesics much more complicated even for Riemannian manifolds \cite{KMBOOK,MM}.

In the first section of this paper,  we give the various definitions for a sub-Riemannian structure on a Banach manifold $M$. We consider horizontal distributions given by $\Delta=\xi(\h)$, where $\h$ is a vector bundle over $M$ and $\xi:\h\rightarrow TM$ a smooth vector bundle morphism. The metric $g$ is directly defined on $\h$ instead of $\Delta$. A curve $q(\cdot)$ is then horizontal if it satisfies an equation of the form
$$
\dot{q}(t)=\xi_{q(t)}u(t)
$$ 
for some control $u(t)\in \h_{q(t)}$ for almost every $t$. This definition allows to consider dense distributions, and the infinite-dimensional equivalent to rank-varying distributions.

In the Section 2, we study the case of approximate and exact controllability. In particular, we see that the natural extension of the Chow-Rashevski theorem gives approximate controllability, and that it is in general impossible to expect exact controllability. However, we do introduce so-called \textit{strong} Chow-Rashevski conditions, which let us prove the following result.
\begin{theos}
 Assume that around any point of $M$, there are $r$ horizontal vector fields $X_1,\dots, X_r$ and $k\in\N\setminus\{0\}$ such that any tangent vector can be written as a span of iterated Lie brackets of at most $k$ horizontal vector fields among the $X_i$, and one other horizontal vector field.
 
Then any two points can be joined by a horizontal curve. Moreover, the sub-Riemannian distance induces on $M$ a topology that is coarser than the intrinsic manifold topology.
\end{theos}

Then, in Section 3, we investigate the sub-Riemannian geodesics. We discuss the appearance of \textit{elusive geodesics}, which cannot be characterized by a Hamiltonian equation, and therefore prevent the proof of a PMP. This is due to the fact that the differential of the \textit{endpoint map} (which is the smooth map that to a control $u(\cdot)$ associates the final point of the corresponding horizontal curve) may have dense image.

We do however obtain the following partial converse to a PMP.
\begin{propos}
Fix $t\mapsto q(t)$ a horizontal curve with control $u(\cdot)$. Assume that there exists $t\mapsto p(t)\in T^*_{q(t)}M\setminus\{0\}$ such that $(q(\cdot),p(\cdot))$ satisfies
\begin{equation}\label{eq:int}
\left\lbrace\begin{aligned}
\dot{q}(t)&=\partial_ph(q(t),p(t)),\\
\dot{p}(t)&=-\partial_qh(q(t),p(t)),
\end{aligned}\right.
\end{equation}
with $h(q,p)=\max_{u\in\h_q}(p(\xi_qu)-\frac{1}{2}g_q(u,u))$. Then $q(\cdot)$ is a critical point of the sub-Riemannian action with fixed endpoint.

If $(q(\cdot),p(\cdot))$ satisfies
$$
\left\lbrace\begin{aligned}
\dot{q}(t)&=\partial_pH^0(q(t),p(t),u(t)),\\
\dot{p}(t)&=-\partial_qH^0(q(t),p(t),u(t)),\\
0&=\partial_uH^0(q(t),p(t),u(t)),
\end{aligned}\right.
$$
with $H^0(q,p,u)=p(\xi_qu)$, then $q(\cdot)$ is a \textit{singular curve}, that is, a critical point of the endpoint map.
\end{propos}
The next step is to prove that integral curves of the Hamiltonian flow associated to $h$ are indeed geodesics. However, if $g$ is a weak metric, $h$ may only be defined on a dense sub-bundle of $T^*M$, on which a Hamiltonian flow may not even be defined. This is already a well-known problem in Riemannian geometry, where weak metric may not have Levi-Civita connections. However, under the assumption that $h$ defines a Hamiltonian flow on a smooth dense sub-bundle $\tau M\subset T^*M$, we prove curves that follow this flow do project to geodesics.
\begin{theos}
Let $\tau M$ be a smooth dense sub-bundle of $T^*M$ on which $h$ is well-defined and admits a $\mathcal{C}^2$ symplectic gradient with respect to the restriction to $\tau M$ of the canonical symplectic form on $T^*M$. These assumptions are always true for strong structures. 

Then integral curves of this symplectic gradient, which are curves that locally satisfy
$$\left\lbrace\begin{aligned}
\dot{q}(t)&=\partial_ph(q(t),p(t)),\\
\dot{p}(t)&=-\partial_qh(q(t),p(t)),
\end{aligned}\right.$$
are local geodesics of the sub-Riemannian structures.
\end{theos}

From there, many questions are still unanswered, the most important of which would be to find a good way to characterize the so-called elusive geodesics.

\section{Banach sub-Riemannian geometry}

 The most common sub-Riemannian structure on a manifold $M$ is given by the restriction of a Riemannian metric to a certain smooth distribution of subspaces $\Delta\subset TM$ of the tangent bundle, the so-called  \textit{horizontal distribution}. However, it can be useful to allow the rank of this distribution to change, which is why we will use the point of view of rank-varying distributions of subspaces \cite{SRBOOK,ABCG}. In this section, we introduce relative tangent spaces and use them to define sub-Riemannian geometry on infinite dimensional manifolds.

\subsection{Notations} 
 
For the rest of this section, fix $M$ be a connected smooth Banach manifold of class $\mathcal{C}^\infty$, modelled on a Banach space $\mathrm{B}$. For any vector bundle $\mathcal{E}\rightarrow M$ over $M$ with typical fiber $\mathrm{E}$, we denote $\Gamma(\mathcal{E})$ the space of smooth sections of $\mathcal{E}$. Finally, for $I=[a,b]$ an interval of $\R$, we define
\begin{equation}\label{eq:h1l2}
H^1\times L^2(I,\mathcal{E})=\{(q(\cdot),e(\cdot)):I\rightarrow\mathcal{E}\mid q(\cdot)\in H^1(I,M),\ e(\cdot)\in L^2(I,\mathrm{E})\text{ in a trivialization}\}.
\end{equation}
Here $H^1(I,M)$ denotes the space of curves that are of Sobolev class $H^1$ in local coordinates. Any such curve is continuous \cite{KMBOOK}, which lets us define a smooth trivialization of $\mathcal{E}$ above it. 

\subsection{Definitions}

Let us start by defining {relative tangent spaces}.

\begin{defi} A \textbf{relative tangent space} on $M$, is a couple $(\h,\xi)$, with $\h$ a smooth Banach vector bundle $\pi^\h:\h\rightarrow M$, with fibers isomorphic to a fixed Banach vector space $\mathrm{H}$, and $\xi:\h\rightarrow TM$ is a smooth vector bundle morphism. Such a couple $(\h,\xi)$ is also called an \textit{anchored vector bundle} in the litterature.

The corresponding \textbf{horizontal distribution} is given by the image $\Delta=\xi(\h)\subset TM$.
\end{defi}


Now let us fix $(\h,\xi)$ a relative tangent bundle on $M$, and $\Delta=\xi(\h)$ the corresponding distribution of horizontal subspaces.


\begin{defi} A \textbf{horizontal vector field} is a vector field $X\in \Gamma(TM)$ that is everywhere tangent to $\Delta$. Equivalently, $X$ is a horizontal vector field when there exists a section $u\in \Gamma(\h)$ such that
$$
\forall q\in M,\quad X(q)=\xi_qu(q).
$$
A curve $q(\cdot)$ of class of Sobolev class $H^1$ on $M$ defined on an interval $I$ is said to be a \textbf{horizontal curve} if we can find a lift $t\mapsto u(t)\in\h_{q(t)}$ of $q(\cdot)$ to $\h$ such that $(q(\cdot),u(\cdot))\in H^1\times L^2(I,\h)$ and
$$
\dot{q}(t)=\xi_{q(t)}u(t),\quad \text{a.e.}\ t\in I.
$$ 
The couple $(q(\cdot),u(\cdot))$ is a \textbf{horizontal system}, with $u$ the \textbf{control} and $q$ the \textbf{trajectory} of the system.
\end{defi}

Endowing $\h$ with a Riemannian metric, we obtain a sub-Riemannian structure which will allow the definition of sub-Riemannian length, action and distance in the next section. Much like in the Riemannian case, one distinguishes two types of metrics: weak and strong metrics.

\begin{defi}
A \textit{weak sub-Riemannian structure} on $M$ is a triple $(\h,\xi,g)$ where $(\h,\xi)$ is a relative tangent space on $M$, and $g:\h\times\h\rightarrow \R$ is a smooth positive definite symmetric bilinear form on each fiber $\h_q$.

The structure is said to be a \textit{strong} when both topologies coincide. In this case, $g$ defines a Hilbert product on each fiber, making $\h$ into a Hilbert bundle.

\end{defi}

\begin{example} Let $\h$ be a distribution of closed subspaces on a Banach manifold $M$, and $g$ be a weak Riemannian metric on $M$. Let $i$ be the inclusion map $\mathcal{H}\rightarrow TM$. Then $(\mathcal{H},i,g_{\vert \h})$ is a regular sub-Riemannian metric. 

A curve $t\mapsto q(t)$ of Sobolev class $H^1$ is horizontal when
$$
\dot{q}(t)\in \h_{q(t)},\quad \text{a.e.}\ t.
$$
The case where $\mathcal{H}$ admits a smooth and closed complement was partially studied in \cite{GMV}, in the more general setting of convenient spaces.
\end{example}

\begin{example}
Take $M=H\times A(H)$, with $(H,\left<\cdot,\cdot\right>)$ a Hilbert space and $A(H)$ the set of Hilbert-Schmidt skew-symmetric linear operators on $H$. Then let $\h=M\times H$ and, for $(q,A)\in M$, define $\xi_{q,A}: H\rightarrow T_qM\simeq H\times A(H)$ by
$$
\xi_{q,A}(u)=(u,\frac{1}{2}q\wedge u)\in H\times A(H),
$$
where the operator $q\wedge u$ satisfies $q\wedge u(w)=\left<q,w\right>u-\left<u,w\right>q$. The bilinear form $g$ can be defined simply by $g(u,v)=\left<u,v\right>$. This is a strong sub-Riemannian structure. Horizontal curves $t\mapsto (q(t),A(t))$ satisfy
$$
\dot{q}(t)=u(t),\quad \dot{A}(t)=\frac{1}{2}q(t)\wedge u(t),\quad \text{a.e.}\ t,\quad u\in L^2(0,1;H).
$$
Horizontal vector fields are of the form
$$
X(q,A)=(u(q,A),\frac{1}{2}q\wedge u(q,A)),\quad u:M\rightarrow H.
$$
\end{example}

\begin{example}\label{ex:diff}
Let $M=\mathcal{D}^s(N)$ the topological group of diffeomorphisms of Sobolev class $H^s$, $s\geq d/2+1$, of a $d$-dimensional compact manifold $N$. Endow $N$ with a smooth relative tangent space $(\mathcal{H}_N,\xi_N)$. This group is a smooth Hilbert manifold, with tangent space at $\varphi$ given by $\Gamma^s(TN)\circ\varphi$ \cite{EM,O,S}, with $\Gamma^s(TN)$ the space of $H^s$-vector fields on $N$.

Consider the relative tangent space $(\mathcal{D}^s(N)\times\Gamma^s(\mathcal{H}_N,\xi)$, with
$$
\xi_\varphi(u)(x)=\xi_{N,\varphi(x)}(u(\varphi(x)).
$$
In other words, a vector $v\in T_\varphi\mathcal{D}^s(N)$ is horizontal if and only if $v\circ\varphi^{-1}$, which is a horizontal vector field of $N$ of class $H^s$.

In this case, a horizontal curve $t\mapsto \varphi(t)$ is just the flow of a time-dependent horizontal vector field of Sobolev class $H^s$.

It is important to note that this relative tangent space is not smooth, as it is simply continuous with respect to $\varphi$. However, it is possible to find an equivalent relative tangent space (i.e., such that horizontal curves are the same) that is smooth. See \cite{AT} and the last section of this paper for more details, and \cite{ATY,ATY2,AMY} for applications of such structures to shape analysis.
\end{example}

\begin{defi}
The \textbf{orbit} $\mathcal{O}_{q_0}$ of a point $q_0$ in manifold $M$ endowed with a relative tangent space is the set of all points $q$ of $M$ that can be connected to $q_0$ by a horizontal curve. 

The structure is said to be \textbf{approximately controllable} from $q_0$ if $\mathcal{O}_{q_0}$ is dense in $M$. It is said to be \textbf{controllable} (or to have the \textbf{exact controllability} property) if $\mathcal{O}_{q_0}=M$ for some $q_0$.
\end{defi}
Note that $q_0\in \mathcal{O}_{q_0}$, and that if $q\in\mathcal{O}_{q_0}$, then $\mathcal{O}_{q_0}=\mathcal{O}_{q}$. In finite dimensions, the well-known Chow-Rashevski theorem provides easily checkable sufficient condition for controllability.
\begin{theo}[Chow-Rashevski, \cite{MBOOK}]\label{th:c-r}
Let $M$ be a connected finite dimensional manifold with a smooth relative tangent space. Also assume that for any $(q,v)\in TM$, there exists horizontal vector fields $X_1,\dots,X_r$ such that
$
v=[\dots[X_1,X_2],X_3],\dots,X_r](q),
$
with $[\cdot,\cdot]$ the usual Lie bracket on smooth vector fields of $M$. Then any two points of $M$ can be connected by a horizontal curve.
\end{theo}
This theorem was improved by Sussmann's orbit theorem, see \cite{SUSS}. We will give more details in Section \ref{sec:cont}.
%
%
%
%

\subsection{Length, energy and distance}

Let $I\subset \R$ be an interval and $(\h,\xi,g)$ be a weak sub-riemannian structure on a Banach manifold $M$. 

\begin{defi}\label{cost} The \textit{action} and \textit{length} of a horizontal system $(q,u):I\rightarrow\mathcal{H}$ is respectively defined by
$$
A(q,u)=\frac{1}{2}\int_Ig_{q(t)}(u(t),u(t))dt\quad\text{and}\quad \int_I\sqrt{g_{q(t)}(u(t),u(t))}dt.
$$
\end{defi}

\begin{remark}
Another possibility is to directly define the sub-Riemannian (semi-)norm of a horizontal vector $X=\xi_q(u)\in T_qM$: the linear map $\xi_q$ defines on its image $\xi_q(\mathcal{H}_q)$ a seminorm $n:\xi_q(\mathcal{H}_q)\rightarrow \R^+$ by
$$
n_q(w)^2=\inf_{u\in \h_q,\ \xi_qu=w}g_q(u,u).
$$
If $q:I\rightarrow M$ is a horizontal curve, its normal length and action can be respectively defined by
$$
L(q)=\int_I n_{q(t)}(\dot{q}(t))dt,\quad\text{and}\quad A(q)=\frac{1}{2}\int_In_{q(t)}(\dot{q}(t))^2dt.
$$
However, the normal action and length may not be smooth when $\xi$, so the action of horizontal systems as given in Definition \ref{cost} is better suited to the study of geodesics.
\end{remark}

%

\begin{defi} The sub-Riemannian distance $d(q_0,q_1)$ between two points $q_0, q_1\in M$ is defined by
$$
d(q_0,q_1)=\inf_{\substack{(q,u)\in L^2(0,1;\h),\\ (q,u)\ \mathrm{horizontal},\\
q(0)=q_0,\ q(1)=q_1}} L(q,u).
$$
\end{defi}
Just like in the Riemannian case, this is a semi-distance, though it may have infinite values when there is no horizontal curve between $q_0$ and $q_1$. Note that this definition implies the following characterization of the orbit of a point $q_0$ in $M$:
$$
\mathcal{O}_{q_0}=\{q\in M\mid d(q_0,q)<+\infty\}.
$$
\begin{lem}\label{dist}
The map $d:M\times M\rightarrow\R\cup\{+\infty\}$ is a semidistance, possibly with infinite values. 

\noindent
When the sub-Riemannian structure is strong, $d$ is always a true distance, that is, $d(q_0,q_1)=0$ if and only if $q_0=q_1$. Moreover, the topology it induces on $M$ is finer than the intrinsic manifold topology.
\end{lem}
\begin{proof}
That $d$ is a semi-distance comes from the basic properties of horizontal systems. Indeed, reversals and concatenation of horizontal systems are also horizontal systems, so the symmetric property and triangular inequality are trivial.

Now assume the metric $g$ is strong and let $q_0\in M$. We work in a coordinate neighbourhood $U$ centered at $q$ such that $\h_{\vert U}\simeq U\times \mathrm{H}$, that can be identified with a small open ball $B_0$ of radius $\varepsilon>0$ of the Banach space $\mathrm{B}$ on which $M$ is modeled. In this chart, since $\xi$ is a smooth vector bundle morphism and $q\mapsto g_q$ is a smooth family of Hilbert norms on $\mathrm{H}$, there exists $c>0$ such that
$$
\forall (q,u)\in B_0\times\mathrm{H},\quad \Vert\xi_q(u)\Vert_{\mathrm{B}}\leq c\sqrt{g_q(u,u)}.
$$
Now, for a horizontal system $(q,u):I=[a,b]\rightarrow\h$ starting at $q(a)=q_0$, let 
$$
t_e=\inf\{t\in[a,b]\mid \Vert q(t)-q_0\Vert_{\mathrm{B}}\in M\setminus B_0\}.
$$
Then $t_e>a$ and
$$
\varepsilon\leq \Vert q(t_e)-q_0\Vert_{\mathrm{B}}=\Vert q(t_e)-q(a)\Vert_{\mathrm{B}}\leq \int_a^{t_e}\Vert\xi_{q(t)}(u(t))\Vert_{\mathrm{B}}dt
\leq c\int_a^{t_e} \sqrt{g_{q(t)}(u(t),u(t))}dt.
$$
In particular, the sub-Riemannian distance between $q_0$ and the sphere $\{\Vert q-q_0\Vert_{\mathrm{B}}=\varepsilon\}$ is no less than $c\varepsilon>0$. Therefore, the triangle inequality shows that if $q_1$ in $M$ does not belong to the ball $B_0$, then $d(q_0,q_1)\geq c\varepsilon>0$. On the other hand, if $q_1\neq q_0$  does belong to $B_0$, then any horizontal curve with endpoints $q_0$ and $q_1$ has length greater than $c\Vert q-q_0\Vert_{\mathrm{B}}$, so $$d(q_0,q_1)\geq c\Vert q-q_0\Vert_{\mathrm{B}}>0.$$

In the end, $q_1\neq q_0$ implies that the sub-Riemannian distance between $q_0$ and $q_1$ is positive, hence the sub-Riemannian distance is a true distance for strong sub-Riemannian manifolds.
\end{proof}

\subsection{Definitions}
Fix a sub-Riemannian Banach manifold $(M,\h,\xi,g)$ modelled on a Banach space $\mathrm{B}$.
\begin{defi}
A \textbf{local geodesic} is a horizontal system $(q,u):I\rightarrow M$ such that, for every $t_0 \in I$, and for every $t_1>t_0$ with $t_1-t_0$ small enough, there is an open neighbourhood $U$ of $q([t_0,t_1])$ such that  any horizontal system $(q',u'):I'\rightarrow U$ with endpoints $q(t_0)$ and $q(t_1)$ satisfies
$$L((q,u)_{\vert[t_0,t_1]})\leq L(q',u').$$
It is a \textbf{geodesic} if we simply have, for $t_0$ and $t_1$ close enough,
$$
L((q,u)_{\vert[t_0,t_1]})=d(q(t_0),q(t_1)),
$$
and a \textbf{minimizing geodesic} if its total length is equal to the distance between its endpoints.

We will also use the same term to describe the trajectory $q(\cdot)$ of such control system.
\end{defi}
This distinction between local geodesic and plain geodesics is necessary for weak structures even in infinite dimensional Riemannian manifolds (see \cite{KMBOOK} for example). However, when the metric is strong, all local geodesics are actually geodesics. This is a trivial consequence of the proof of Lemma \ref{dist}.
\begin{remark}
If a horizontal system $(q,u)$ minimizes the action $A(q,u)$ among controls whose trajectories have the same endpoints, then the trajectory $q$ is a minimizing geodesic. This is a trivial consequence of the Cauchy-Schwartz inequality.
\end{remark}

The existence of minimizing geodesics between two points is a difficult question in infinite dimensions. For example, even for strong Riemannian Hilbert manifolds, metric completeness does not imply geodesic completeness.

\begin{example} Consider $X=l^2(\N)$ the space of square-summable sequences, and let $M$ be the ellipsoid given by
$$
\left\{(x_n)\in l^2,\ \sum_{n=0}^\infty \frac{x_n^2}{(1+\frac{1}{n+1})^2}=1\right\}
$$
equipped $M$ with the Riemannian metric inherited from the ambient space. $M$ will be complete for the Riemannian distance but there will be no minimizing geodesic between $(\sqrt{2},0,\dots)$ and $(-\sqrt{2},0,\dots)$.
\end{example}

\section{Exact and approximate controllability}\label{sec:cont}

The first problem when considering sub-Riemannian geometry is that of controllability: can we get from any starting point $q_0$ to any target point $q_1$ of $M$ through horizontal curves. In finite dimensions, the Chow-Rashevsky theorem \cite{MBOOK} (and its more general version, Sussmann's orbit theorem \cite{SUSS,CTBOOK}) gives a nice sufficient condition for the controllability of the structure: it is controllable when the iterated Lie brackets of horizontal vector fields span the entire tangent space. Moreover, the ball-box theorem also gives precise estimates on the sub-Riemannian distance in this case, showing that it is topologically equivalent to the intrinsic manifold topology of $M$. 

We will see that these conditions are unreasonnable to expect in the case of infinite dimensional manifolds. All we can usually expect to have dense orbits, that is, approximate controllability. We will however give some natural, stronger conditions that do ensure exact controllability.

For the rest of this section, unless stated otherwise, $M$ is a sub-Riemannian Banach manifold endowed with a sub-RIemannian structure $(\mathcal{H},\xi,g)$.

\subsection{Finite dimensions: the Chow-Rashevski Theorem}

We identify the horizontal distribution $\Delta$ with $\Gamma(TM)$, the $\mathcal{C}^\infty(M)$-module of all horizontal vector fields on $M$ of class $\mathcal{C}^\infty$. By induction, we define the nondecreasing sequence of $\mathcal{C}^\infty(M)$-modules $(\Delta_i)_{i\in\N}$ by
$$
\Delta^0=\{0\},\quad\Delta^1=\Delta,\quad \Delta^{i+1}=\Delta^i+[\Delta,\Delta^i],
$$
where $[\cdot,\cdot]$ denotes the Lie bracket for smooth vector fields on $M$. Then $\mathcal{L}=\cup_{i\in\N} \Delta^i$ is the \textit{Lie algebra} of vector fields generated by $\Delta$. 

\begin{remark}
Note that any $\mathcal{C}^\infty(M)$-module $\mathcal{E}$ of vector fields can be identified to a subset $\mathcal{E}'$ of $TM$ given by 
$$ \mathcal{E}'=\bigcup_{q\in M}\mathcal{E}_q\subset TM,\quad
\mathcal{E}_q=\lbrace X(q)\mid X\in \mathcal{E}\rbrace\subset T_qM.
$$
\end{remark}

\begin{remark}
This definition is valid for both finite and infinite dimensional manifolds.
\end{remark}

\begin{defi}
We say that the sub-Riemannian structure satisfies the \textit{Chow-Rashevski property} at $q\in M$ when $\mathcal{L}_q=T_qM$.
\end{defi}

For the rest of the section, we assume that $M$ is finite dimensional. We can now re-state the Theorem \ref{th:c-r} as follows.

\begin{theo}[\textbf{Chow-Rashevski theorem in finite dimensions} \cite{CTBOOK,SRBOOK,MBOOK}]\label{th:CR}
Assume that the finite dimensional sub-Riemannian manifold $M$ satisfies the Chow-Rasevski property at every point $q\in M$. Then that structure is exact controllable. 
\end{theo}

\begin{remark}
A more precise result is given by Sussmann's Orbit Theorem \cite{SUSS}. It states that each orbit $\mathcal{O}_q$ is an immersed submanifold such that $\mathcal{L}_q\subset T_q\mathcal{O}_q$. Sussmann also proved that, if the structure is analytic, we actually have $\mathcal{L}_q= T_q\mathcal{O}_q$.
\end{remark}

The proof of Chow-Rashevski's theorem can be refined to give local estimates on the sub-Riemannian distance. Indeed, fix $q_0\in M$,  assume $T_{q_0}M=\mathcal{L}_{q_0}=\Delta^k_{q_0}\neq \Delta^{k-1}_{q_0}$, and let $r_i=\dim(\Delta^{i}_{q_0})-\dim(\Delta_{q_0}^{i-1})$ for $i=1,\dots,k$ so that $r_1+\dots+r_k=\dim(M)$.

\begin{theo}[\textbf{Ball-box theorem in finite dimensions} \cite{SRBOOK,MBOOK}]\label{th:bb}
Around such a $q_0\in M$, there are coordinates $q=(x_1,\dots,x_k)$, with $x_k\in \R^{r_i}$ around $q_0$ such that for some $C>0$ such that
$$
\frac{1}{C}\sum_{i=1}^{k}\vert x_i\vert^{2/i}\leq d(q_0,q)^2\leq {C}\sum_{i=1}^{k}\vert x_i\vert^{2/i}.
$$
\end{theo}
In particular, the topology induced by the sub-Riemannian distance coincides with the intrinsic manifold topology of $M$.

%
%
%
%
%
%
%
%
%
%

The exact statement of the Chow-Rashevski and ball-box theorems are both open problems in infinite dimensional manifolds. Moreover, even if such a result existed, it would not be as useful: it is very rare to have $\mathcal{L}_q=T_qM$, simply because $\mathcal{L}_q$ is usually dense, but almost never closed. This is expected, intuitively, because the Lie algebra generated by $\Delta$ is constructed in an algebraic way as an indefinitely increasing union of brackets of horizontal vector fields. Let us give an in-depth example, which will also be useful for seeing what happens when studying geodesics, in the next section.

\subsection{An Example: the $\ell^2$-product of Heisenberg Groups}
\label{sec:heis}
We take an in-depth look at the problem of controllability in a very simple example of infinite dimensional sub-Riemannian manifold, the $\ell^2$-product of Heisenberg groups.

\subsubsection{The 3-Dimensional Heisenberg Group} 

The Heisenberg group is the simplest case of finite dimensional sub-Riemannian manifold. The manifold itself is $\mathbb{H}=\R^3$, and the horizontal space at $q=(x,y,z)$ is spanned by
$$
X(q)=(1,0,-\frac{x}{2})=\frac{\partial}{\partial x}-\frac{y}{2}\frac{\partial}{\partial z},\quad Y(q)=(0,1,\frac{x}{2})=\frac{\partial}{\partial y}+\frac{x}{2}\frac{\partial}{\partial z},
$$
which are orthonormal for the metric. Horizontal curves $q(\cdot)=(x(\cdot),y(\cdot),z(\cdot))$ therefore satisfy
$$
\dot{x}(t)=u(t),\quad \dot{y}(t)=v(t),\quad \dot{z}(t)=\frac{1}{2}(v(t)x(t)-u(t)y(t)),\quad u,\,v\in L^2(I;\R),
$$
with action
$$
A^\mathbb{H}(q(\cdot))=\frac{1}{2}\int_I (u(t)^2+v(t)^2)dt.
$$
Since $[X,Y](q)=\frac{\partial}{\partial z}$, so that $X,Y,[X,Y]$ span the tangent bundle, any two points can be connected by a horizontal curve, and the sub-Riemannian distance satisfies the ball-box estimates of Theorem \ref{th:bb}
$$
\frac{1}{C}(x^2+y^2+\vert z\vert)\leq d^\mathbb{H}(0,(x,y,z))\leq C (x^2+y^2+\vert z\vert)
$$
for some fixed $C>0$.
\vspace{2mm}
%

\subsubsection{The $\ell^2$-product of Heisenberg groups} \label{sec:proheiscont}

We now consider the Hilbert manifold $M=\ell^2(\N,\R^3)$ the space of square-summable sequences $q=(q_n)_{n\in\N}=(x_n,y_n,z_n)_{n\in\N}$ of $\R^3$. We define on it the sub-Riemannian structure generated as $q$ by the Hilbert frame
$$
X_n(q)=\frac{\partial}{\partial x_n}-\frac{y_n}{2}\frac{\partial}{\partial z_n},\quad Y_n(q)=\frac{\partial}{\partial y_n}+\frac{x_n}{2}\frac{\partial}{\partial z_n}.
$$

\paragraph{Lie Algebra.} We denote by $\mathcal{L}$ the Lie algebra of smooth vector fields generated by horizontal vector fields. Now, we have $Z_n:=[X_n,Y_n]=\frac{\partial}{\partial z_n}$, so that the horizontal vector fields give a Hilbert-spanning frame of $TM$. In other words, any tangent vector can be written as an \textit{infinite} linear combination with $\ell^2$ coefficients of brackets of horizontal vector fields. However, they do not span it as a vector field.

For example take the two horizontal vector fields
$$
X=\sum_{n\in\N} a_{n}X_n,\quad Y=\sum_{n\in\N}b_{n}Y_n,\quad \sum_{n\in\N}a_n^2+b_n^2<+\infty.
$$
Then
$$
[X_1,X_2]=\fr \sum_{n\in\N} a_nb_n\frac{\partial}{\partial z_n}.
$$
But, as a product of $\ell^2$-sequences, $(a_nb_n)_{n\in\N}$ actually belongs to the dense subspace of absolutely summable sequences $\ell^1(\N,\R)\subset \ell^2(\N,\R)$. 
\vspace{2mm}

More generally, one easily checks that any tangent vector $v\in T_qM$ at $0$ belongs to the $\mathcal{L}$ if and only if it can be written
$$
v=\sum_{n\in \N} a_nX_n(q)+b_nY_n(q)+c_nZ_n(q), \quad \sum_{n\in\N} a_n^2+b_n^2<+\infty,\quad \sum_{n\in\N} \vert c_n\vert\leq +\infty.
$$
Therefore, $\mathcal{L}$ is only \textit{dense} in $M$.

\paragraph{Orbit of $0$.} Let us describe the orbit of $0$ in $M$. A curve $t\mapsto q(t)=(q_n(t))_{n\in\N}$ is horizontal if and only if each curve $t\mapsto q_n(t)=(x_n(t),y_n(t),z_n(t))\in \R^3\simeq \mathbb{H}$ is horizontal for the 3-dimensional Heisenberg group. Moereover, its action is given by
$$
A(q(\cdot))=\sum_{n\in\N} A(q_n(\cdot)).
$$
Consequently, the sub-Riemannian distance between $0$ and $q=(x_n,y_n,z_n)_{n\in\N}$ is given by
$$
d(q,q')^2=\sum_{n=0}^{+\infty} d_{\mathbb{H}}\big((x_n,y_n,z_n),(x'_n,y'_n,z'_n)\big)^2,
$$
with $d_{\mathbb{H}}$ denoting the sub-Riemannian distance on the Heisenberg group $\mathbb{H}$ as described in the previous section. But we know from Theorem \ref{th:CR} that
$$
\exists\, C>0,\ \forall (x,y,z)\in \mathbb{H},\quad \frac{1}{C}(x^2+y^2+\vert z\vert)\leq d_{\mathbb{H}}(0_{\mathbb{H}},(x,y,z))^2\leq C(x^2+y^2+\vert z\vert).$$
In particular $d(0,q)$ is finite if and only if $(x_n,y_n)_{n}\in l^2(\N,\R^2)$ while $(z_n)_n\in l^1(\N,\R)$. In other words, 
$$ 
\mathcal{O}_{0}=l^2(\N, \R^2)\times l^1(\N, \R)\subset l^2(\N, \R^3).
$$
Moreover, the topology on $\mathcal{O}_{0}$ induced by the sub-Riemannian distance actually coincides with the usual Banach space topology of $l^2(\N, \R^2)\times l^1(\N, \R)$, and we get a dense orbit in $M$. We also lost the Hilbert topology.

We can actually get even worse, even in the simple case of $\ell^2$ products of Carnot groups.
\subsubsection{A Non-Locally Convex Topology}\label{sec:nonconv}
Slightly complicating our example slightly, if we take $M=l^2(\N,\R^4)$ as an infinite product of the Engel group $\mathbb{E}$ \cite{MBOOK} (or any step-3 or higher Carnot group), we start getting even less satisfactory topologies. Indeed, we will once more get $\mathcal{O}_{0}=l^2(\N,\mathbb{E})$. But on the Engel group, 
$$
\frac{1}{C}(x^2+y^2+\vert z\vert+\vert w\vert^{2/3})\leq d_{\mathbb{H}}(0_{\mathbb{H}},(x,y,z,w))^2\leq C(x^2+y^2+\vert z\vert+\vert w\vert^{2/3}).
$$
In other words, the sub-Riemannian distance on $\mathcal{O}_{{ 0}}$ is equivalent to the usual quasi-distance on $l^2(\N, \R^2)\times l^1(\N, \R)\times l^{2/3}(\N, \R)$, whose topology is \textit{not locally convex}.

\subsection{Approximate Controllability}
As we just saw, the conditions for Chow-Rashevski's theorem are very rarely satisfied. However, it is much more common for $\mathcal{L}_q$ to be \textit{dense} in $T_qM$. In this case, as was proved in \cite{DS,HX,KM}, we do have approximate controllability.

\begin{theo}[\cite{DS,HX,KM}]
Assume that $M$ is a Banach manifold, and that $\mathcal{L}_q$ is dense in $T_qM$ for every $q$ in $M$. Then each orbit $\mathcal{O}_q$ is dense in $M$, so that the structure is  approximate controllable.
\end{theo}
\begin{remark}
This is actually true even for so-called convenient manifolds (i.e., manifolds modelled on convenient vector spaces, see \cite{KMBOOK}), as shown in \cite{KM}.
\end{remark}
The proof uses the fact that for a proper closed subset $F$ of a Banach space $B$, there is a cone $C=\lbrace q_0+tq,\ q\in K, 0\leq t\leq 1\rbrace$ with nonempty interior, vertex $q_0\in F$ and such that $C\cap F=\lbrace q_0\rbrace$. This in turn implies that there are no $\mathcal{C}^1$-curves in $F$ starting at $q_0$ with initial velocity in the interior of $C$. But taking $F$ to be the closure of an orbit, one can easily build a curve from $q_0$ with initial velocity given by any $v\in \mathcal{L}_{q_0}$. Hence, if $F\neq M$, we get that $\mathcal{L}_{q_0}$ is not dense in $B$.

\subsection{Exact controllability and strong Chow-Rashevski property}\label{sec:crinf}

The question of exact controllability is much more complex, even in simple casesas we saw in Section \ref{sec:heis}, and we do not often have exact controllability.

However, under a stronger hypothesis, and by working a little harder, we can still obtain it.
\subsubsection{The Strong Chow-Rashevski Property}
\begin{defi}
The sub-Riemannian structure is said to satisfy the \textbf{strong Chow-Rashevski property} at $q\in M$ if there exists fixed horizontal vector fields $X_1,\dots,X_r\in \Delta$ and a fixed positive integer $k$ such that
\begin{equation}\label{stcr}
T_qM=\Delta_q+\sum_{i\in\{1,\dots,r\}} [\Delta,X_i]_q+\dots+\sum_{I\in\{1,\dots,r\}^k}[\Delta,X_I]_q,
\end{equation}
where, for simplicity, we denoted
$$
X_I=[X_{i_j},[\dots,[X_{i_2},X_{i_1}]\dots],\quad I=(i_1,\dots,i_j)\in \{1,\dots,r\}^j,\quad j\in\{1,\dots,k\}.
$$
\end{defi}
In this case, we can adapt the proof of the finite dimentional Chow-Rashevski theorem to the infinite dimensional context.
\paragraph{Examples.} Before we state our result, let us give a few examples of infinite dimensional sub-Riemannian manifolds that satisfy this property.
\begin{example}
If the horizontal distribution has finite condimension everywhere, then the Chow-Rashevski condition and the strong Chow-Rashevski condition are equivalent. This is the case for an infinite dimensional Heisenberg group $\mathbb{H}^\infty=\ell^2(\N,\R^2)\times\R$, with horizontal vector fields spanned by
$$
X_n(q)=X_n(x_n,y_n,z)=\frac{\partial}{\partial x_n}-\fr y_n\frac{\partial}{\partial z},\quad Y_n(q)=Y_n(x_n,y_n,z)=\frac{\partial}{\partial y_n}+\fr x_n\frac{\partial}{\partial z},\quad n\in\N.
$$
Here $T_qM=\Delta_q+[\Delta_q,X_n]=\Delta_q+[\Delta_q,Y_n]$ for any integer $n$.
\end{example}

\begin{example}
Consider $M=\R\times\ell^2(\N,\R^2)$, with Hilbert basis of horizontal vector fields given by
$$
X(q)=X(x,y_n,z_n)=\frac{\partial}{\partial x},\quad Y_n(q)=\frac{\partial}{\partial y_n}+x\frac{\partial}{\partial z_n}.
$$
We have $T_qM=\Delta_q+[\Delta_q,X]$. Indeed, at $q=0$ for example, any tangent vector can be written 
$$
v=a\frac{\partial}{\partial x}(0)+\sum_{n\in\N}b_n\frac{\partial}{\partial y_n}(0)+\sum_{n\in\N}c_n\frac{\partial}{\partial z_n}(0),
$$
for some $a\in \R,\ b,c\in \ell^2(\N,\R).$ Then, letting $Y=aX+\sum_{n\in\N}b_nY_n$ and $Y'=\sum_{n\in\N}c_nY_n$, we have
$$
v=Y(0)+[Y',X](0).
$$
\end{example}

\begin{example}
It was proved in \cite{AC} that the sub-Riemannian structure on the group of diffeomorphisms $\mathcal{D}^s(N)$ of a compact $d$-dimensional sub-Riemannian manifold $N$ defined in Example \ref{ex:diff} has exact controllability. It was proved in \cite{AT} that, after some work, this structure also satisfies the strong Chow-Rashevski condition, and estimates on the corresponding sub-Riemannian distance were given.
\end{example}

\subsubsection{Statement and Proof of the Theorem}
\begin{theo}
Assume the strong Chow-Rashevski property is satisfied at $q_0\in M$ for some fixed vector fields $X_1,\dots,X_r\in \Delta$, and $k$ the smallest integer such that \eqref{stcr} is satisfied. Then $\mathcal{O}_{q_0}$ contains a neighbourhood of $q_0$, and the topology induced by the sub-Riemmanian distance is coarser than its intrinsic manifold topology.

As a consequence, if the metric is strong, the two topologies coincide.
\end{theo}
\begin{remark}
Since $M$ is connected, an immediate consequence is that if the strong Chow-Rashevski property is satisfied at every point of $M$, then $\mathcal{O}_{q_0}=M$ and we have exact controllability.
\end{remark}
\begin{proof} 

We work on a small neighbourhood $V_0$ of $q_0$, that we identify to an open subset of a Banach space $\mathrm{B}$, and on which we have a trivialization $V_0\times \mathrm{H}\,{\simeq}\, \h_{\vert V_0}$. In this trivialization, for any $u\in H$, we define the smooth vector field $q\mapsto X_u(q)=\xi_q(u)$. We can assume that each $X_i$ is of the form $X_i=X_{u_{0,i}}$ (changing the trivialization if necessary. We also denote $t\mapsto\varphi^t(u):V_0\rightarrow V_0$ the corresponding flow on $V_0$. 

Fix a positive integer $N$, that we will assume to be as big as needed. The Cauchy-Lipshitz theorem with parameters shows that there exists a smaller neighbourhood $V_1\subset V_0$ of $q_0$, and a neighbourhood $U_0$ of $0$ in $\mathrm{H}$, such that for any $u_1,\dots,u_{N}$ in $U_0$, $q$ in $V_1$ and $t_1,\dots,t_{N}$ in $[-1,1]$,
$$
\varphi^{t_1}(u_1)\circ\cdots\circ\varphi^{t_{N}}(u_{N})(q)\in V_0.
$$
Note that this mapping is smooth in all variables $t$, $u$ and $q$, as composition of flows of smooth vector fields that depend smoothly on a parameter.

Multiplying each $X_i$ by an appropriate constant if necessary, we can assume that $X_i=X(u_{0,i}),$ with $u_{0,i}\in U_0$. To simplify notations, for the rest of the proof, we will denote $\varphi_i=\varphi(u_{i,0}),\ i\in \{1,\dots,r\}$, and, for $j\in\{1,\dots, k\}$ and $I=(i_1,\dots,i_j)\in \{1,\dots,r\}^j$, we denote the iterated Lie bracket
$$
X_I=[X_{i_j},[\dots,[X_{i_2},X_{i_1}]\dots].
$$
We then define for each such $I$ and $t\in [0,1]$ the smooth mapping $\varphi_I^t:V_1\rightarrow V_0$ by
$$
\varphi_I^t(q)=\varphi_{i_j}^{t}\circ\cdots\circ\varphi_{i_1}^{t}.
$$
We can now define, for $u\in U_0$ and $t\in[0,1]$,
$$
\phi_I^{t}(u,q):=\varphi^{1}(u)\circ\varphi_{I}^{-t}\circ\varphi^{1}(u)\circ\varphi_I^{t}(q).
$$
Note that $(t,u)\mapsto \phi_I^t(u)$ is smooth, and that its range is included in $\mathcal{O}_{q_0}$ as a concatenation of $2k+2$ horizontal curves. Moreover, the usual formulas for commutators of flows yields
$$
\phi_I^{t}(u,q)=q+t^{i}[X(u),X_I](q)+o(t^{i}u)
$$
as $u,t\rightarrow 0$, for fixed $q$. Therefore, for fixed $q$, the mapping $U_0\mapsto V_0$
\begin{equation}\label{eq:Phi}
\Phi_I(u,q)=\phi_I^{\Vert u\Vert_{\mathrm H}^{1/{i+1}}}\left(\frac{u}{\Vert u\Vert_{\mathrm{H}}^{\frac{i}{i+1}}},q\right)
\end{equation}
is smooth outside of 0. Moreover, since $u\mapsto X(u)$ is linear, $\Phi_I$ is also  $\mathcal{C}^1$ around 0 with first order limited development
$$
\Phi_I(u,q)=q+[X(u),X_I](q)+o(u).
$$
From there, we easily see that the mapping 
$$\begin{aligned}
&\Phi:&U_0\times U_0^r\times\cdots\times U_0^{kr}&\rightarrow& V_0\\
&&{\bf u}=(u_{I_{i}})_{\substack{i=0,\dots,k,\quad \\I_i\in \{1,\dots,r\}^i}}&\mapsto&\left(
\underset{\substack{i=0,\dots,k, \\I_i\in \{1,\dots,r\}^i}}{\bigcirc}\Phi_I(u_{I_i})\right)(q_0)
\end{aligned}
$$
is of class $\mathcal{C}^1$ near 0, with
$$
\Phi({\bf u})=q_0+\sum_{\substack{i=0,\dots,k, \\I_i\in \{1,\dots,r\}^i}}[X( u_{I_i} ),X_I](q_0)+o({\bf u}).
$$
Then $d\Phi(0):H^{1+\dots+r^k}\rightarrow T_{q_0}M$ is onto as a consequence of the strong Chow-Rashevski condition. Hence, its range contains a neighbourhood of $q_0$. But since $\Phi({\bf u})$ is obtained by taking the endpoint of a concatenation of horizontal curves, we see that its range is included in the orbit of $q_0$. Consequently, the orbit of $q_0$ does contain a neighbourhhod of $q_0$.

\vspace{2mm}
Moreover, let $C_1=\max_{i=1,\dots r} (g_{q_0}(u_{i,0},u_{i,0})$, where we recall that $X_i=X(u_{i,0})$, and $C_2$ such that $g_q(u,u)\leq C\Vert u\Vert_{\mathrm H}^2$ on $V_0\times\mathrm{H}$ (reducing $V_0$ if necessary). Finally, let $C=\max(1,C_1,C_2)^2.$ Then we see that any $\Phi_I(u,q)$, $I\in\{1,\dots,r\}^i,\ i=0,\dots,k,$ is obtained by taking the endpoint of $2i+2$ curves of energy less than 
$
C\Vert u\Vert_{\mathrm H}^{2/{i+1}}.
$
Consequently, we get one side of the ball-box estimates:
$$
d(q_0,\Phi({\bf u}))^2\leq C'\sum_{\substack{i=0,\dots,k, \\I_i\in \{1,\dots,r\}^i}} \Vert u_{I_i}\Vert_{\mathrm H}^{2/{i+1}}.
$$
In particular, any sub-Riemannian ball around $q_0$ includes a neighbourhood of $q_0$, so that the topology induced by the distance is coarser than the intrinsic manifold topology.
\end{proof}
\begin{remark}
It would be much preferable to obtain estimates on the distance of the form
$$
d(q_0,\Phi({\bf u}))^2\leq C'\sum_{\substack{i=0,\dots,k, \\I_i\in \{1,\dots,r\}^i}} g_{q_0}(u_{I_i},u_{I_i})^{1/{i+1}}.
$$
This is obviously true in the strong case. However, in the weak case, we would need to replace each instance of $\Vert u\Vert_{\mathrm H}$ by $\sqrt{g_{q_0}(u,u)}$ in the formula \eqref{eq:Phi} for $\Phi_I$. But then the term
$$
\frac{u}{\left(\sqrt{g_{q_0}(u,u)}\right)^{i/i+1}}
$$
may not go to 0 as $u$ goes to zero, which prevents the rest of the proof from working. This version of the ball-box estimates is therefore an open conjecture.
\end{remark}

\begin{remark}
The converse inequality (for the strong case) is still open. The proof in finite dimensions uses the concept of privileged coordinates, which are much harder to find in infinite dimensions.
\end{remark}

\begin{cor}
If the strong Chow-Rashevski condition is satisied at some point $q_0$ of $M$, then the Banach space $\mathrm{B}$ on which $M$ is modelled is a quotient of products of $\mathrm{H}$. In particular, if the metric is strong, then the norm of $\mathrm{B}$ is equivalent to a Hilbert norm. 
\end{cor}
The $\ell^2$-product of Heisenberg group restricted to $\ell^2(\N,\R^2)\times\ell^1(\N,\R)$ does not satisfy the strong Chow-Rashevski condition, although it satisfies the plain Chow-Rashevski condition. This indicates that the strong version of the condition may still hold in infinite dimensions. The proof would most likely be very different, a require completely new mathematical tools.

\section{Geodesics and the Hamiltonian Geodesic Flow}

The purpose of this section is to study geodesics on infinite dimensional sub-Riemannian manifolds. We first recall the various possible definitions of geodesics in infinite dimensions. Then, we define a manifold structure on certain subsets of horizontal systems in order to investigate first order conditions for a such a system to prject onto a geodesic, and recall some properties of the canonical weak symplectic form on the cotangent bundle of a manifold. 

We then finally investigate first order conditions for a horizontal curve to be a geodesic. We will in particular see that no first order \textit{necessary} condition can be given in general. However, we will give \textit{sufficient} conditions for a curve to be a critical point of the action with fixed endpoints, and \textit{sufficient} conditions for a curve to be abnormal. We will also see that in spite of this, there is still a Hamiltonian flow of geodesics in the strong case, and we will give conditions for the existence of such a flow for weak sub-Riemannian structures that specialize to the well-known corresponding conditions on Riemannian manifolds. After that, we go back to the case of $\ell^2$-product of Heisenberg groups, in order to highlight the problems and differences that appear in infinite dimensions.

But first, let us establish a few notations and definitions of symplectic geometry that will be necessary to define the Hamiltonian flow.
\subsection{Symplectic Gradient and Partial Symplectic Gradient}
\label{sec:symp}
Recall that a 2-form $\omega$  on a Banach manifold $N$ is said to be \textit{weak symplectic} if it is closed, and if the linear mapping $v\in T_xN\mapsto \omega_x(v,\cdot)\in T_x^*N$ is one-to-one for each $x$ in $N$. 

We now fix a Banach manifold $M$, modelled on a Banach space $\mathrm{B}$. Let $\omega$ be the canonical weak symplectic form on $T^*M$. Recall that $\omega$ is a closed 2-form on $T^*M$ defined in canonical coordinates by
$$
\omega_{q,p}(\delta q,\delta p;\delta q',\delta p')=\delta p(\delta q')-\delta p'(\delta q).
$$

\begin{remark}
For us, ``in canonical coordinates" will mean in a chart $\Psi:T^*M_{\vert U}\rightarrow \psi(U)\times \mathrm{B}^*$ of the form $\Psi(q,p)=(\psi(q),d\psi(q)_*p)$, with $\psi:U\mapsto\psi(U)\subset\mathrm{B}$ a diffeomorphism. We then identify $T^*M_{\vert U}\simeq\psi(U)\times \mathrm{B}^*$ so that $(q,p)\simeq (\psi(q),d\psi(q)_*p)$ for readability.
\end{remark}
\paragraph{Symplectic gradient of a function.} Take a smooth function $f:T^*M\rightarrow\R$, and let $(q,p)\in T^*M.$ We say that $f$ admits a \textit{symplectic gradient} at $(q,p)$ if there exists a vector $\nabla^\omega f(q,p)\in TT^*M$ such that
$$
df(q,p)=\omega(\nabla^\omega f(q,p),\cdot).
$$
In infinite dimensions, it is well-known that not every smooth function admits a symplectic (unless $\mathrm{B}$ is reflexive), see \cite{KMBOOK} for example. 

Now the partial derivative of $f$ along the fiber $\partial_pf(q,p)$ is defined intrinsically. It belongs to $(T_q^{*}M)^*=T_q^{**}M$. Denote by $j$ the canonical dense inclusion $T_qM\hookrightarrow T_q^{**}M$. If $\partial_pf(q,p)$ belongs to the image of $j$, it can then be identified to the vector $j^{-1}(\partial_pf(q,p))\in T_qM$, which we also denote $\partial_pf(q,p)$. In this case, $f$ does admit a symplectic gradient $\nabla^\omega f(q,p)$ at $(q,p)$, as this gradient is given in canonical coordinates by the formula
$$
\nabla^\omega f(q,p)=(\partial_pf(q,p),-\partial_qf(q,p))\in T_qM\times T_q^*M\simeq T_{(q,p)}T^*M.
$$

\paragraph{Symplectic Partial Gradient of a Function.} We now consider a vector bundle $\mathcal{E}$ on $M$, and denote by $T^*M\oplus_M \mathcal{E}$ the vector bundle with fiber $(T^*M\oplus_M \mathcal{E}h)_q=T^*_qM\times  \mathcal{E}_q$.
 Let $f:T^*M\oplus_M \mathcal{E}\rightarrow\R$ be a smooth function, and fix $(q,p,e)\in T^*M\oplus_M \mathcal{E}$. Assume that $\partial_pf(q_0,p_0,e_0)$ belongs to $j(T_qM)$, with $j:TM\rightarrow T^{**}M$ the canonical inclusion.
 \vspace{2mm}
 
Then, if $\partial_e f(q_0,p_0,e_0)=0$, $f$ admits a unique \textit{symplectic partial gradient} $\nabla^\omega f(q_0,p_0,e_0)\in T_{(q,p)}^*TM$ at $(q_0,p_0,e_0)$, in the following sense: for any smooth mapping $e:(q,p)\mapsto e(q,p)\in\mathcal{E}_q$ with $e(q_0,p_0)=e_0$, we have
$$
d(f(q,p,e(q,p)))_{q=q_0,p=p_0}=\omega(\nabla^\omega f(q_0,p_0,e_0),\cdot).
$$
Indeed, in a local trivialisation of $\mathcal{E}$ and canonical coordinates on $T^*M$ near $(q_0,p_0)$, we have
$$\begin{aligned}
d(f(q,p,e(q,p)))_{q=q_0,p=p_0}&=\partial_qf(q_0,p_0,e_0)+\partial_pf(q_0,p_0,e_0)+ \underset{=0}{\underbrace{\partial_ef(q_0,p_0,e_0)}}.de(q_0,p_0)\\
&=\partial_qf(q_0,p_0,e_0)+\partial_pf(q_0,p_0,e_0).
\end{aligned}
$$
This gradient is therefore given in canonical coordinates by 
$$
\nabla^\omega f(q_0,p_0,e_0)=(\partial_p f(q_0,p_0,e_0),-\partial_qf(q_0,p_0,e_0)).
$$

\paragraph{Restriction of $\omega$ to a dense sub-bundle.} Let $\tau  M\hookrightarrow T^*M$ be a smooth, dense sub-bundle of $T^*M$.
We have the following trivial (but crucial) lemma.
\begin{remark}
We use a $*$ in $\tau M$ to emphasize that we are in a sub-bundle of a dual space. We \textit{do not} require $\tau M$ to be the dual bundle of some vector bundle $\tau M$.
\end{remark}
\begin{lem}
The restriction of $\omega$ to $\tau  M$, i.e., its pull-back through the inclusion map $\tau  M\rightarrow T^*M$, is a weak symplectic form on $\tau  M$. We still denote it $\omega$.
\end{lem}
\begin{proof}
The 2-form $\omega$ on $\tau M$ is closed as the pull-back of a closed 2-form. It is not degenerate because of the density of $\tau M$ in $T^*M$.
\end{proof}

A smooth map $f:\tau M\rightarrow\R$ then admits a symplectic gradient at $(q,p)\in\tau M$ if and only if, in canonical coordinates, the following conditions are satisfied:
\begin{enumerate}
\item $f$ admits a symplectic gradient at $(q,p)$ in $T^*M$, i.e., the partial derivative of $f$ along the fiber $\partial_pf(q,p)\in T^{**}_qM$ belongs to the image of the canonical embedding $j:T_qM\hookrightarrow T_q^{**}M$, through which it can be identified to a vector $\nabla_pf(q,p)\in T_qM$, and
\item the symplectic gradient of $f$ belongs to $T\tau_{(q,p)}^*M\subset TT^*_{(q,p)}M$. In other words, in a canonical chart, $\partial_qf(q,p)$ (which, in the chart, is an element of $T^*_qM$) actually belongs to the dense subspace $\tau _qM$.
\end{enumerate}
In this case, in those coordinates, we can indeed write 
$$
\nabla^\omega f(q,p)=(\nabla_pf,-\partial_qf)\in T_{(q,p)}\tau M\subset T_{(q,p)}T^*M,
$$
which is a stable property under a change of canonical coordinates.
\vspace{2mm}

We can now investigate the geodesics of a sub-Riemannian structure.

\subsection{Manifold Structure on the Set of Horizontal Systems}
Let $M$ be a Banach manifold endowed with a smooth sub-Riemannian structure $(\h,\xi,g)$. Let us fix $I=[0,1]$ to simplify notations. We denote
$$
\Omega^\mathcal{H}=\{(q,u)\in H^1\times L^2(I,\h)\mid (q,u)\ \text{horizontal}\}
$$
the set of all horizontal systems. We also define, for $q_0,q_1$ in $M$,
$$
\Omega^\mathcal{H}_{q_0}=\{(q,u)\in \Omega^\mathcal{H}\mid q(0)=q_0\},\quad \Omega^\mathcal{H}_{q_0,q_1}=\{(q,u)\in \Omega^\mathcal{H}\mid q(0)=q_0,\ q(1)=q_1\}.
$$
To give conditions for a curve to be a geodesic, it is natural to study critical points of the action among horizontal systems with fixed endpoints (that is, critical points of $A$ on $\Omega^\h_{q_0,q_1}$). For this, we need to put a manifold structure on the space of horizontal systems.

When $M$ is a $d$-dimensional manifold, the space $H^1\times L^2(I,\mathcal{H})$ as defined in \eqref{eq:h1l2} is a smooth Banach manifold (see for example the appendix of \cite{MBOOK}), as an $L^2(I,\mathrm{H})$-fiber bundle over $H^1(I,M)$, which is a Hilbert manifold, see \cite{EM,ES}.

This is no longer the case in general when $M$ is a Banach manifold. Indeed, to build the atlas for $H^1(I,M)$ one needs a smooth \textit{local addition}: a smooth diffeomorphism $F_1\subset TM\rightarrow F_2\subset M\times M$, from $F_1$ neighbourhood of the zero section onto $F_2$ neighbourhood of the diagonal, such that any $(q,0)\in U$ is mapped to $(q,q)$ \cite{KMBOOK}.

However, since the concept of geodesic is a local one (or, in the case of minimizing geodesics, only concerns curves without self-intersections), we can study a horizontal system on small enough time intervals that the trajectory stays on a domain on which we can trivialize both $\h$ and $TM$. For our purpose in this section, we can therefore assume that $M$ is an open subset of $\mathrm{B}$, and that $\h\simeq M\times \mathrm{H}$. In this case, $H^1\times L^2(I,\h)$ is just an open subset of $H^1(I,M)\times L^2(I,\mathrm{H})$. Then, we have the following result, first proved in \cite{ATY}[Lemma 3] but whose proof we include for the sake of completeness.
\begin{propo}\label{prop:str}
We keep the notations and assumptions of the previous discussion. Fix $q_0$ in $M$. Then the space $\Omega^\h_{q_0}$ of horizontal systems whose trajectories start at $q_0$  is a smooth submanifold of $H^1\times L^2(I,\mathcal{H}_{\vert U})$, diffeomorphic to an open subset of $L^2(I,\mathrm{H})$ through
$$
u\in \mathcal{U}\subset L^2(I,\mathrm{H})\mapsto (q(u),u)\in H^1\times L^2(I,\mathcal{H}_{\vert U})
$$
The  \textbf{trajectory map} $u\mapsto q(u)$ is obtained by solving the Cauchy problem
$$
q(0)=q_0,\quad \dot{q}(t)=\xi_{q(t)}u(t).
$$
\end{propo}
\begin{proof}
Define $C: H^1\times L^2(I,\mathcal{H})\rightarrow L^2(I,\mathrm{B})$ by
$$
C(q,u)(t)=\dot{q}(t)-\xi_{q(t)}u(t).
$$
Note that since $M$ is an open subset of $\mathrm{B}$ in our case, we have $TM\simeq M\times\mathrm{B}$. 

Then $\Omega^h_{q_0}=C^{-1}(\{0\})$. Now fix $(q,u)\in\Omega^\h_{q_0}$. The operator $\partial_q C(q,u):H^1(I,\mathrm{B})\rightarrow L^2(I,\mathrm{B})$ is given by
$$
(\partial_q C(q,u).\delta q)(t)=\dot{\delta q}(t)-(\partial_{q}\xi_{q(t)}u(t)).\delta q(t),\quad \delta q\in H^1(I,\mathrm{B}).
$$
This is obviously a Banach isomorphism: for any $a\in L^2(I,\mathrm{B})$,
$$
\delta q(0)=0,\quad (\partial_q C(q,u).\delta q)(t)=a(t),\quad a.e. t\in I
$$
is just a linear Cauchy problem, and hence admits a unique global solution $\delta q=\partial_qC(q,u)^{-1}a$. The implicit function theorem then states that $\Omega^h_{q_0}=C^{-1}(\{0\})$ is the graph of a smooth mapping $u\mapsto q(u)$ from an open subset of $L^2(I,\mathrm{H})$ onto $H^1(I,M)$. 
\end{proof}
Consequently, we will often identify $L^2(I,\mathrm{H})$ with $\Omega^\h_{q_0}$, and a control $u$ with the corresponding horizontal system $(q(u),u)$.

\subsection{Endpoint Mapping and Critical Points of the Action} 
The endpoint mapping is defined in the following trivial corollary of Proposition \ref{prop:str}.
\begin{cor} 
The so-called \textbf{Endpoint map} $E:(q,u)\in\Omega_{q_0}\mapsto q(u)(1)$ is smooth. Its derivative at $u$ in the direction $\delta u\in L^2(I,\mathrm{H})$ is equal to $\delta q(1)$, where $(q,\delta q)\in H^1(I,TM)$ and $\delta q$ is obtained by solving the linear Cauchy problem
$$
\delta q(0)=0,\quad \dot{\delta q}(t)=\partial_q(\xi_{q(t)}u(t)).\delta q(t)+\xi_{q(t)}\delta u(t).
$$
\end{cor}

Note that $\Omega^\h_{q_0,q_1}=E^{-1}(\{q_1\}).$ Therefore, when looking for geodesics between $q_0$ and $q_1$, one attempts to solve the smooth constrained optimal control problem of minimizing
$$
A(u)=A(q,u)=\fr\int_0^1g_{q(t)}(u(t),u(t))dt
$$
among all $(q,u)$ in $\Omega_{q_0}^\h$ such that $E(q,u)=q_1$.

\paragraph{The three types of sub-Riemannian geodesics.} Before we move on, we need to discuss the apparition in infinite dimension of a new type of geodesics, called \textit{elusive geodesics}. They were introduced for the first time in \cite{AT}. 

Fix a minimizing geodesic $(q,u)=(q(u),u)$, which we identify with the corresponding optimal control $u$. We know that the smooth map $F=(A,E):\Omega_{q_0}^\h\rightarrow \R\times M$ must have a derivative that is not onto. We have two possibilities:
\begin{enumerate}
\item The range of $dF(u)$ has positive codimension in $\R\times T_{q_1}M$, that is, its closure is a proper subset of $T_{q_1}M$.
\item The range of $dF(u)$ is a proper dense subset of $\R\times T_{q_1}M$. This can only happen when $M$ is infinite dimensional.
\end{enumerate}
Using a cotangent viewpoint, these condition can be reformulated as:
\begin{enumerate}
\item There exists $(\lambda,p_1)\in \{0,1\}\times T^*_{q_1}M\setminus\{(0,0)\}$ such that
$$
\lambda dA=dE(u)^*p_1,
$$
where $dE(u)^*:T^*_{q_1}M\rightarrow L^2(I,\mathrm{H})^*$ is the adjoint map of $dE(u)$. Depending on the value of $\lambda$, this splits into two subcases:\begin{enumerate}
\item \textit{The normal case:} $\lambda=1$, which gives $ dA=dE(u)^*p_1$. This corresponds to a \textit{normal geodesic}, from which we will derive the Hamiltonian flow later in the section.
\item \textit{The abnormal case:} $\lambda=0$, which gives $ 0=dE(u)^*p_1$ and $p_1\neq0$. This implies that $u$ is a \textit{singular control} (i.e., a critical point of the endpoint map), and we say that $(q,u)$ is an \textit{abnormal geodesic}. While there is no characterization of abnormal geodesics, even in finite dimensions, there is a nice Hamiltonian characterization yielding all singular controls \cite[Chapter 5]{MBOOK}. We will give the infinite dimensional version of this result.
\end{enumerate}
\item $dF(u)^*$ is one-to-one, which is no different from the case of non minimizing curves. This gives no useful Hamiltonian characterization. We say that $(q,u)$ is an \textit{elusive geodesic}.
\end{enumerate}
This problematic second is the reason why there is no Pontryagin principle in infinite dimensions \cite{LY}. It is actually a very common occurence in infinite dimensional sub-Riemannian geometry. For example, any curve in the $\ell^2$ product of Heisenberg groups with no constant component is elusive.
\begin{remark}
As discussed in \cite{A2,ATY,AT}, an interpretation of this is that the topology induced by the sub-Riemannian distance is much finer than the manifold topology. Hence, there are not enough Lagrange multipliers $p_1$. However, if the sub-Riemannian structure can be restricted to a smooth dense embedded submanifold $M'\subset M$, that is, a manifold modelled on a Banach space with dense and continuous inclusion in $\mathrm{B}$, such that $M'$ contains $\mathcal{O}_{q_0}$, then $T^*_{q_1}M\varsubsetneq T_{q_1}^*M'$, and we obtain additional multipliers, which turns some elusive curves into additional normal and abnormal extremums that are more easily characterized.

This is the case when $M=\ell^2(\N,\R^3)$ is the $\ell^2$-product of Heisenberg groups , where one can restrict the structure to $M'=\ell^2(\N,\R^2)\times \ell^1(\N,\R)$, whose cotangent space at $0$ is $\ell^2(\N,\R^2)\times \ell^\infty(\N,\R)$ which is much bigger than $T_0^*M=\ell^2(\N,\R^3)$.

The question of finding the ``right" tangent bundle, that is, one for which there are no elusive geodesics, is open an would probably require more powerful and innovative tools to solve. For example, the structure described in Section \ref{sec:nonconv} seems to indicate that the correct dense submanifold to which we should restrict the structure would be $\ell^2(\N,\R^2)\times\ell^1(\N,\R)\times\ell^{2/3}(\N,\R)$, which is not even locally convex and therefore has a dual space that is too small.
\end{remark}

So there are no first order necessary conditions for a control to yield a geodesic in infinite dimensions. However, there is a partial converse to this result which does remains true. First of all, we say that a control $u$ is a critical point of the action with fixed enpoint if, for any $\mathcal{C}^1$ family of controls $s\in(-\varepsilon,\varepsilon)\subset\R\mapsto u_s$ such that $q(u_s)(1)=q_1$ for each $s$ and $u_0=u$, we have
$$
\partial_s(A(u_s))_{\vert s=0}=0.
$$
Note that any geodesic is such a critical point. Then the following result is immediate.
\begin{lem}\label{lem:conv}
Fix a control $u\in \mathcal{U}$.
\begin{enumerate}
\item If there exists $p_1\in T_{q_1}^*M$ such that $dA(u)=dE(u)^*p_1$, then $u$ is a critical point of the action with fixed endpoints.
\item If there exists $p_1\in T_{q_1}^*M\setminus\{0\}$ such that $0=dE(u)^*p_1$, then $u$ is a singular control.
\end{enumerate}
\end{lem}
To obtain a workable version of these conditions, we need to compute, for any control $u$ and $p_1\in T^*_{q_1}M,$ a good expression for $\lambda dA(u)-dE(u)^*p_1$, with $\lambda=0,1$. This is given by the Hamiltonian formulation.

\subsection{Hamiltonian Formulation}
We keep the same notations as in the previous section. We define the  Hamiltonian $H^\lambda:T^*M\oplus_M\h\rightarrow \R$ of the problem by the smooth expression
$$
H^\lambda(q,p,u)=p(\xi_qu)-\frac{\lambda}{2}g_q(u,u).
$$
Here, $T^*M\oplus_M\h$ is the vector bundle above $M$ with fiber at $q$ given by $T^*_qM\times \h_q$. 

Notice that the (intrinsically defined) partial derivative of $H^\lambda$ in $p$ satisfies
$$
\partial_pH^\lambda(q,p,u).\delta p=\delta p(\xi_qu),\quad \delta p\in T^*qM,
$$
so that $\partial_pH^\lambda(q,p,u)$ can be identified to $\xi_qu$ through the canonical inclusion $T_qM\rightarrow T^{**}_qM$. Consequently, $H^\lambda$ admits a symplectic gradient $\nabla^\omega H^\lambda(q,p,u)$ (see Section \ref{sec:symp} for the appropriate definitions), given in canonical coordinates by
$$
\nabla^\omega H^\lambda(q,p,u)=(\partial_pH(q,p,u),-\partial_qH(q,p,u)).
$$
\begin{propo}
Fix a control $u$ in a local space of controls $\mathcal{U}$ such that $(q,u)=(q(u),u)\in\Omega_{q_0,q_1}^\h$. Then
\begin{equation}\label{eq:form1}
\lambda dA(u)=dE(u)^*p_1,\quad (\lambda,p_1)\in \{0,1\}\times T^*_{q_1}M\setminus\{(0,0)\},
\end{equation}
if and only if there exists $t\in I\mapsto p(t)\in T^*_{q(t)}M$ of class $H^1$ such that $p(1)=p_1$ and, for almost every $t$ in $t$,
\begin{equation}\label{eq:form2}
\begin{aligned}\exists t\in I\mapsto p(t)\in T^*_{q(t)}M \text{ of class }H^1,\ p(1)=p_1,\text{ and, a.e.}\ t\in I,\\
\left\lbrace
\begin{aligned}
0\ \ \ \quad&=\partial_u H^\lambda(q(t),p(t),u(t)),\\
(\dot{q}(t),\dot{p}(t))&=\nabla^\omega H^\lambda(q(t),p(t),u(t)).
\end{aligned}\right.\qquad\qquad
\end{aligned}
\end{equation}
In this case, $(q,u)$ is automatically a critical point of the action with fixed endpoints when $\lambda=1$, and a critical point of the endpoint map (i.e., $u$ is a singular control) when $\lambda=0$.
\end{propo}
As mentionned in the previous section, the converse to the last statement is not true. See Section \ref{.} for examples.
\begin{remark}
Note that \eqref{eq:form1}, and even the definition of singular controls and critical points of the action, requires the local viewpoint we adopted (i.e., that $M$ is an open subset of $\mathrm{B}$). 

However, that is not the case for \eqref{eq:form2}. Indeed, even though $\nabla^\omega H^\lambda(q(t),p(t),u(t))$ usually depends on a trivialization of $\h$, its value becomes intrisic when $\partial_uH^\lambda(q(t),p(t),u(t))$ vanishes (see Section \ref{sec:symp} for a proof). We can then simply use condition \eqref{eq:form2} to identify geodesics and singular curves even in the global viewpoint.
\end{remark}
The proof was given in \cite{AT} for the special case of strong structures on groups of diffeomorphisms. The general proof is almost the same.
\begin{proof}
The proof is the same as in finite dimensions. Fix $u\in L^2(0,1;\mathrm{H})$ $q$ the corresponding trajectory, and $(\lambda,p_1)\in \{0,1\}\times T^*_{q_1}M\setminus\{(0,0)\}$. Take $\delta u\in L^2(0,1;\mathrm{H})$. We have $dE(u).\delta u=\delta q(1)$, with $\delta q\in H^1(I,\mathrm{B})$ solution of
$$
\delta q(0)=0,\quad \dot{\delta q}(t)=\partial_q(\xi_{q(t)}u(t))+\xi_{q(t)}\delta u(t).
$$
Hence
\begin{equation}\label{eq:int1}
\lambda dA(u)-dE(u)^*p_1=\int_0^1\left(\lambda g_{q(t)}(u(t),\delta u(t))+\frac{\lambda}{2}\partial_q(g_{q(t)}(u(t),u(t))).\delta q(t)\right)dt - p(1)(\delta q(1)),
\end{equation}
where $t\mapsto p(t)\in T^*_{q(t)}M$ solves the linear Cauchy problem
$$
p(1)=p_1,\quad \dot{p}(t)=-\partial_qH^\lambda (q(t),p(t),u(t))=-\partial_q(\xi_{q(t)}u(t))^*p(t)+\frac{\lambda}{2}\partial_q g_{q(t)}(u(t),u(t)).
$$
But we see that $\frac{\lambda}{2}\partial_q(g_{q(t)}(u(t),u(t)))=\dot{p}(t)+\partial_q(\xi_{q(t)}u(t))^*p(t)$ so that a term $\dot{p}(t)\delta q(t)$ appears in the right-hand side of \eqref{eq:int1}. An integration by part on this term, and the fact that $\delta q(0)=0$ will yield
$$
\lambda dA(u)-dE(u)^*p_1=\int_0^1
\left(
\lambda g_{q(t)}(u(t),\delta u(t))+p(t)(\partial_q(\xi_{q(t)}u(t)).\delta q(t))-p(t)(\dot{\delta q}(t))\right)dt.
$$
But replacing $\dot{\delta q}(t)$ with $\partial_q(\xi_{q(t)}u(t))+\xi_{q(t)}\delta u(t)$ finally gives us
$$
\lambda dA(u)-dE(u)^*p_1=\int_0^1
\left(
\lambda g_{q(t)}(u(t),\delta u(t))-p(t)(\xi_{q(t)}\delta u(t))\right)dt
=
-\int_0^1\partial_uH^\lambda(q(t),p(t),u(t))dt
.
$$
In particular,
$$
\lambda dA(u)-dE(u)^*p_1\quad \Longleftrightarrow \partial_uH^\lambda(q(t),p(t),u(t))=0\ \text{a.e.}\ t\in I.
$$
\end{proof}

\subsection{Hamiltonian Geodesic flow}
We now investigate the existence of a Hamiltonian flow for the normal geodesics. We will find that, much like in the Riemannian case, strong structures always admit such a flow, while additional assumptions are required for weak sub-Riemannian manifold.

\paragraph{The strong case.} We assume for now that the sub-Riemannian structure is strong. In this case, because each $g_q$ is a Hilbert product,  the equation
$$
\partial_uH^1(q,p,u)=0=\xi_q^*p-g_q(u,\cdot)
$$
has a unique solution $u(q,p)=G_q^{-1}\xi_q^*p$ for any $(q,p)\in T^*M$ (Riesz representation theorem). Here $G_q^{-1}$ is the smooth inverse of the smooth vector bundle isomorphism $G:u\in\h_q\mapsto g_q(u,\cdot)\in \h^*$, also called the musical operator. This lets us define the \textit{normal Hamiltonian} of the structure $h:T^*M\rightarrow\R$ by
\begin{equation}\label{eq:normham}
h(q,p)=H^1(q,p,u(q,p))=\fr g_q(u(q,p),u(q,p)).
\end{equation}
Do note that, since $H^1$ is strictly concave in $u$, we can also write $h(q,p)=\max_{u\in\h_q}H^1(q,p,u)$.

Now thanks to the fact that $\partial_uH^1(q,p,u(q,p))=0$, $h$ admits a smooth symplectic gradient given by
$$
\nabla^\omega h(q,p)=\nabla^\omega H^1(q,p,u(q,p)),\quad (q,p)\in T^*M.
$$
This gradient can be integrated into a well-defined smooth local flow that we call the \textit{Hamiltonian geodesic flow.}

\begin{theo}[Hamiltonian geodesic flow: strong case] On a strong sub-Riemannian manifold, the normal Hamiltonian is well-defined, and for any $(q_0,p_0)\in T^*M$, there is a unique maximal solution to the normal Hamiltonian equation
$$
(\dot{q}(t),\dot{p}(t))=\nabla^\omega h(q(t),p(t)).
$$
More importantly, any such solution $t\mapsto (q(t),p(t))$ projects to a \textbf{geodesic} $q(\cdot)$ on $M$.
\end{theo}
We will give the proof for this theorem at the same time as that for the weak case later in the section.

\paragraph{Adapted cotangent sub-bundles.}
Some extra difficulties can appear when the metric is weak. More precisely, the equation $\partial_uH=0$ may not have a solution for every $(q,p)\in T^*M$, so that the normal Hamiltonian may not be defined. Hence, we need to restrict ourselves to a subspace on which it is well-defined. We will need the following definitions.

\begin{defi}
A \textbf{relative cotangent bundle} is a Banach vector bundle $\tau M\rightarrow M$ which admits a smooth, {dense} immersion in $T^*M$.

Such a bundle $\tau M$ is said to be \textbf{adapted} to a sub-Riemannian structure $(\h,\xi,g)$ on $M$ if, for every $q\in M$ and every $p\in \tau_q^*M$, there exists $u(q,p)\in \h_q$ such that
$$
\xi_q^*p=g_q(u(q,p),\cdot).
$$
\end{defi}
In other words, a sub-bundle $\tau M$ is adapted to the structure if the normal Hamiltonian can be define as in \eqref{eq:normham} on $\tau M$Note that such a $u(q,p)$ is always uniquely determined by $q$ and $p$ (and linear in $p$). Indeed, $g_q$ is positive definite, so $u\mapsto g_q(u,\cdot)$ is injective.

\begin{remark}
If $g$ is a strong metric, then $T^*M$ itself is of course adapted to the structure.
\end{remark}

\begin{remark}
In the case of a weak Riemannian structure $(\h,\xi)=(TM,\mathrm{Id}_{TM})$, one can take $\tau M=g(TM,\cdot)$, with the topology of $TM$ itself. It is the only relative cotangent bundle adapted to the structure.
\end{remark}

\begin{remark}
More generally, if $\h$ is a closed subbundle of $TM$ and $g$ the restriction of a weak Riemannian metric to $\h$, one can still take $\tau M=g(TM,\cdot)$. This is the relative cotangent bundle used to find the geodesic equations in \cite{GMV}.
\end{remark}

We now have the following trivial (but crucial) lemma.
\begin{lem}
The restriction of $\omega$ to a dense sub-bundle $\tau M\hookrightarrow T^*M$, i.e., its pull-back through the inclusion map, is a weak symplectic form on $\tau M$. We still denote it $\omega$.
\end{lem}
\begin{proof}
The 2-form $\omega$ on $\tau M$ is closed, since it is the pull-back of a closed 2-form. The fact that it is not degenerate comes fromthe density of $\tau M$ in $T^*M$.
\end{proof}

Let $(\h,\xi,g)$ be a weak sub-Riemannian structure on a Banach manifold $M$, with typical fiber $H$. Let $\tau M$ be an adapted relative cotangent bundle. The restriction of $H^1$ to $\tau M$ is given by
$$
H^1(q,p,u)=p(\xi_qu)-\frac{1}{2}g_q(u,u)=g_q\left(u(q,p)-\frac{1}{2}u,u\right),
$$
so that we can indeed define the normal Hamiltonian on $\tau M$ by
$$
h(q,p)=\max_{u\in \h_q}(H^1(q,p,u))=h(q,p,u(q,p))=\frac{1}{2}g_q(u(q,p),u(q,p)).
$$
The partial derivative of $H$ along the fiber in $\tau M$ is given by
$$
\partial_p h(q,p)(\delta p)=\partial_pH^1(q,p,u(q,p))(\delta p)+\partial_uH^1(q,p,u(q,p))(u(q,\delta p))=\delta p(\xi_qu(q,p)),
$$
since $\partial_uh(q,p,u(q,p))=0$. In particular, we can identify $\partial_ph(q,p)\in \tau^{**}M$ with the tangent vector $\nabla_ph(q,p)=\xi_qu(q,p)\in T_qM$. 

From now on, we make the following important assumptions:
\smallskip

\begin{itemize}
\item[•](A1): The mapping $(q,p)\mapsto u(q,p)$ is smooth on $\tau M$, so that $h$ is also smooth. 

\item[•](A2): For every $(q,p)\in\tau M$, $\partial_qH(q,p)$ belongs to $\tau H$, so that $H$ admits a symplectic gradient. Moreover, this gradient is at least of class $\mathcal{C}^2$ on $\tau M$.

\end{itemize}
The symplectic gradient $\nabla^\omega h(q,p)$ is given in local coordinates by
$$
\nabla^\omega h(q,p)=(\xi_qu,\frac{1}{2}\partial_qg(u,u)-(\partial_q\xi_{q}u)^*p),
$$
where $u=u(q,p)$ is such that $\xi_q^*p=g_q(u(q,p),\cdot)$. Moreover, as a vector field of class $\mathcal{C}^2$, it admits a local flow. Curves induced by that flow are said to satisfy the \textit{Hamiltonian geodesic equation}. Note that the Hamiltonian is constant along such curves.

We can now state our main theorem.

\begin{theo}\label{geodeq}
Let $(\h,\xi,g)$ be a weak sub-Riemannian structure on a Banach manifold $M$. Let $\tau M\hookrightarrow T^*M$ be an adapted relative cotangent bundle such that (A1) and (A2) are satisfied. Let $(\bar q(\cdot),\bar p(\cdot)):I\rightarrow \tau M$ be a curve that satisfies the \textit{Hamiltonian geodesic equation}
$$
(\dot{\bar q}(t),\dot{ \bar{p}}(t))=\nabla^\omega h(\bar q(t),\bar p(t)),\quad t\in I,
$$
which, in local coordinates, gives
$$\left\lbrace
\begin{aligned}
\dot{ \bar{q}}(t)&=\xi_{ \bar{q}(t)}\bar{u}(t),\\
\dot{ \bar{p}}(t)&=\frac{1}{2}\partial_qg_{ \bar{q}(t)}( \bar{u}(t), \bar{u}(t))-(\partial_q\xi_{ \bar{q}(t)} \bar{u}(t)^* )\bar{p}(t),
\end{aligned}\right.
$$
for every $t$ in $I$, where $ \bar{u}(t)= \bar{u}(\bar q(t),\bar p(t))$ is the unique element of $\h_{\bar q(t)}$ such that $\xi_{\bar q(t)}^*\bar p(t)=g_{\bar q(t)}( \bar{u}(t),\cdot)$.

Then the projection $\bar q(\cdot):I\rightarrow M$ is a horizontal curve and a \textbf{local geodesic} for the sub-Riemannian structure. 
\end{theo}

\begin{remark}
When the metric $g$ is strong, $\tau M=TM$ and hypotheses (A1) and (A2) are of course automatically satisfied.  In the weak Riemannian case on the other hand, it is equivalent to the existence of a smooth Levi-Civita connection, see for exemple \cite{MM}. These are therefore very natural hypotheses for this theorem.
\end{remark}

\begin{remark}
The hypothesis that $h$ has a symplectic gradient for the relative cotangent bundle is equivalent to that made in \cite[Theorem 1]{GMV} on the existence of a transpose operator.
\end{remark}

\begin{proof}
First of all, if $\xi_{\bar q_0}^*\bar p_0=\xi_{\bar q(0)}\bar p(0)=0$, then $\bar q$ and $\bar p$ are constant curves, and therefore $\bar q$ is a trivial geodesic, so we can assume $\xi_{\bar q_0}^*\bar p_0\neq0$ and therefore $\bar u_0=\bar u(0)\neq 0$. Now since $\dot{\bar q}(t)=\xi_{\bar q(t)}\bar u(t)$, $ \bar{q}$ is obviously horizontal, so we just need to prove that it is a local geodesic. We can assume that $I=[0,1]$ without loss of generality. The proof uses an idea similar to that of the minimizing property of geodesics in Riemannian geometry, with a few tweaks. We need to prove that for $\varepsilon>0$ small enough, and any horizontal system $(q(\cdot),u(\cdot)):[0,\varepsilon]\rightarrow\h$ such that $q(0)=\bar q_0$ and $q(\varepsilon)=\bar q(\varepsilon)$, we have
$$
L((\bar q,\bar u)_{\vert [0,\varepsilon]}\leq L(q,u).
$$
For this, we will find a \textit{calibration} of $\bar q$: a closed 1-form $\theta$ on a neighbourhood of $\bar q_{\vert [0,\varepsilon]}$ in $M$ such that
\begin{itemize}
\item for every $t\geq0$ small enough,
$$
\theta_{\bar q(t)}(\dot{\bar q}(t))=c\sqrt{h(\bar{q}(t),\bar{p}(t))}\leq c\sqrt{g_{\bar q(t)}(\bar u(t),\bar u(t))},
$$
with $c>0$ a fixed constant, and
\item for every $(q,u)\in\h$ with $q$  close enough to $\bar q_0$,
$$
\vert\theta_q(\xi_qu)\vert\leq c\sqrt{g_{ q}(u,u)}.
$$
\end{itemize}
Indeed, once $\theta$ is found, we know that in a small neighbourhood of $\bar q_0$, it is exact. Hence, for $\varepsilon>0$ small enough, and for any horizontal system $(q(\cdot),u(\cdot))$ in this neighbourhood with $q(0)=\bar q_0$ and $q(\varepsilon)=\bar{q}(\varepsilon)$,
$$
L((\bar q,\bar u)_{\vert [0,\varepsilon]})=\frac{1}{c}\int_0^\varepsilon\theta_{\bar q(t)}(\dot{\bar q}(t))dt=\frac{1}{c}\int_0^\varepsilon\theta_{ q(t)}(\dot{q}(t))dt\leq \int_0^\varepsilon \sqrt{g_{ q(t)}(u(t),u(t))}dt=L(q,u).
$$
We now build this calibration.

We work in a coordinate neighbourhood $U\subset M$ of $\bar q_0$ in $M$ that we identify with a coordinate neighbourhood centered at $q_0$ in the Banach space $\mathrm{B}$ on which $M$ is modelled, so that we can simply write $q(0)=\bar q_0=0$. We also consider a trivialization 
$$
(\tau M\underset{M}{\oplus}\h)_{\vert U}\simeq U\times \mathrm{G}\times\mathrm{H}.
$$
We denote the local $\mathcal{C}^2$-flow of $\nabla^\omega h$ on $\tau M$ by 
\begin{equation}\label{eq:hamflow}
(t,q,p)\mapsto \Phi(t,q,p)=(\Phi_M(t,q,p),\Phi_\tau(t,q,p)
\end{equation}
Note that, in the coordinate neighbourhood $U$, $\Phi(t,0,\bar p_0)=(\bar q(t),\bar p(t))$ for $t$ small enough and $\bar p_0=\bar p(0)$.

The bundle $\tau M$ has smooth dense inclusion in $T^*M$, so we can write  that $\bar p_0$ belongs to $\mathrm{B}^*\setminus\{0\}$, and its kernel $\ker \bar p_0$ is a closed hyperplane of $\mathrm{B}$. Let 
$$
U_0=U\cap \ker \bar p_0.
$$ 
Note that $\bar q_0=0$ belongs to $U_0$, and that this $U_0$ is a neighbourhood of $0$ in $\ker p_0$.

Reducing $U_0$ if needed, we can then define the map $\varphi:]-\varepsilon,\varepsilon[\times U_0\rightarrow U$ of class $\mathcal{C}^2$ by
$$
\varphi(t,q_0)=\Phi_M\left(t,q_0,\sqrt{\frac{h(q_0,\bar p_0)}{h(\bar q_0,\bar p_0)}}\bar p_0\right),
$$
with $\Phi_M$ as in \eqref{eq:hamflow}. Do note that $h(\bar q_0,\bar p_0)=\bar p_0(\xi_{\bar q_0}\bar u_0)=g_{\bar q_0}(\bar u_0,\bar u_0)>0$. The positive number $n(q_0)=\sqrt{\frac{h(q_0,p_0)}{h(\bar q_0,\bar p_0)}}$ is here so that 
$$
h(q_0,n(q_0)\bar p_0)=h(\bar q_0,\bar p_0),\quad q\in U_0.
$$ 
For $q_0\in U_0$, the curve $t\mapsto\varphi(t,q_0)$ is the projection to $M$ of the Hamiltonian flow starting at $q_0$ with initial condition $p(0)=n(q)\bar p_0$. 

\begin{lem}
Reducing $U_0$ and $U$ if necessary, there exists $\varepsilon>0$ such that the mapping $\varphi$ is a local diffeomorphism of $]-\varepsilon,\varepsilon[\times U_0$ onto $U$.
\end{lem}
\begin{proof}
We just need to prove that $d\phi(0,\bar q_0)$ is bijective. For any $\delta q\in \ker p_0$, we have 
$$
\partial_q\varphi(\bar q_0,\bar q_0)\delta q=\partial_q\Phi_M(0,\bar q_0,\bar p_0)\delta q=\delta q,
$$
so $\partial_q\varphi(\bar q_0,\bar q_0)=\mathrm{Id}_{\ker \bar p_0}$. This is because $\Phi$ is the flow of a vector field. Now we just need to prove that $\partial_t\varphi(0,\bar q_0)$ does not belong to $\ker p_0$. But
$$
\partial_t\varphi(0,\bar q_0)=\partial_t\Phi_M(0,\bar q_0,\bar p_0)=\dot{ \bar{q}}(0)=\xi_{\bar q_0}\bar u_0.
$$
Since $\bar p_0(\xi_{\bar q_0}\bar u_0)=h(\bar q_0,\bar p_0)>0$, $\partial_t\varphi(0,\bar q_0)$ does not belong to $\ker \bar p_0$.
\end{proof}
Now, for $q$ in $U$, let $(t(q),q_0(q))=\varphi^{-1}(q)$. This mapping is of class at least $\mathcal{C}^2$, same as $\varphi$. This lets us define on $U$ the one-form 
$$
\theta(q)=\Phi_\tau(t(q),q_0(q),n(q_0(q))\bar p_0)\in \tau _qM\subset \mathrm{B}^*,
$$
with $\Phi_\tau$ defined as in \eqref{eq:hamflow}. In other words, $\theta$ is given by the propagation to $U$ of $q_0\mapsto n(q_0)p_0$ on $U_0$, so that
$$
\theta(\varphi(t,q_0))=p(t),
$$ 
where $t\mapsto(q(t),p(t))$ follows the Hamiltonian flow with initial condition $q(0)=q_0$ and $p(0)=n(q_0)\bar p_0$. Let us prove that $\theta$ is a calibration of $\bar q.$

Fix $q$ in $U$. We can write $\theta(q)=p(t(q))$, where $(q(t),p(t))$ satisfies the Hamiltonian geodesic equations with $(q(0),p(0))=(q_0(q),n(q_0(q))\bar p_0)$. Therefore, 
$$
\vert \theta(q)\xi_qu\vert =\vert p(t(q))\xi_qu\vert =\vert g_q(u(q,p(t(q)),u)\vert \leq \underset{=\sqrt{h(q,p(t(q)))}}{\underbrace{\sqrt{g_q(u(q,p(t(q)),u(q,p(t(q)))}}}\sqrt{g_q(u,u)}.
$$
Since the reduced Hamiltonian is constant along the Hamiltonian flow, we have     
$$
h(q,\theta(q))=h(q,p(t(q)))=h(q_0(q),n(q_0(q))\bar p_0)=h(\bar q_0,\bar p_0).
$$
Letting $c=\sqrt{h(\bar q_0,\bar p_0)}>0$, we get, for any $(q,u)\in\h_{\vert U}$,
$$
\vert\theta_{q}\xi_qu\vert\leq c\sqrt{g_q(u,u)}.
$$
Therefore, if we can prove that $\theta$ is closed, we do have that $\theta$ calibrates $\bar q$. Since $\varphi$ is a diffeomorphism, it is enough to prove the following lemma.

\begin{lem}
We have $\varphi^*\theta=c^2dt$ on $]-\varepsilon,\varepsilon[\times U_0$.
\end{lem}
\begin{proof}
Fix $(t_0,q_0)$ in $]-\varepsilon,\varepsilon[\times U_0$, and let $t\mapsto(q(t),p(t))$ follow the Hamiltonian flow with initial condition $q(0)=q_0$ and $p(0)=n(q_0)\bar p_0$. In particular, $\theta(\varphi(t,q))=p(t).$ We will alos denote $u(t)=u(q(t),p(t))$ the corrsponding control.

Then for every $(\delta t,\delta q)\in\R\times\ker \bar p_0$,
\begin{equation}\label{vtheta}
\begin{aligned}
(\varphi^*\theta)_{(t_0,q_0)}(\delta t,\delta q_0)&=
\theta(\varphi(t_0,q_0))(\partial_t\varphi(t_0,q_0)\delta t+\partial_{q_0}\varphi(t_0,q_0)\delta q_0)\\
&=
\Phi_\tau(t_0,q_0,n(q_0)p_0)(\partial_t\varphi(t_0,q_0)\delta t+\partial_{q_0}\varphi(t_0,q_0)\delta q_0)\\
&=p(t_0)(\partial_t\varphi(t_0,q_0)\delta t)+p(t_0)(\partial_{q_0}\varphi(t,q_0)\delta q_0).
\end{aligned}
\end{equation}
Now recall that $\partial_t\varphi(t,q)=\dot{q}(t)=\xi_{q(t)}u(q(t),p(t)),$ so that for every time $t$,
\begin{equation}\label{vthtat}
\begin{aligned}
p(t)(\partial_t\varphi(t,q))&=\frac{1}{2}g_{q(t)}(u(q(t),p(t)),u(q(t),p(t)))=h(q(t),p(t))\\
&=h(q(0),p(0))=h(q_0,n(q_0\bar p_0)=h(\bar q_0,\bar p_0)=c^2.
\end{aligned}
\end{equation}

Now let us check that $p(t)(\partial_{q_0}\varphi(t,q_0)\delta q_0)=0$ for every $\delta q_0$ in $\ker \bar p_0$ and $t$ in $]-\varepsilon,\varepsilon[$. For small $s>0$ and $t\in]-\varepsilon,\varepsilon[,$ denote $q(s,t)=\varphi(t,q_0+s\delta q_0)$. For each $s$, $t\mapsto q(s,t)$ is horizontal, with associated control $t\mapsto u(s,t)$ that can be taken $\mathcal{C}^2$ in $s$ and such that $g_{q(s,t)}(u(s,t),u(s,t))=2c$ for every $(s,t)$ \footnote{Each $q(s,\cdot)$ is the projection to $M$ of a curve $(q(s,\cdot),p(s,\cdot))$ that follows the Hamiltonian flow with initial condition $(q_0+s\delta q_0,n(q_0+s\delta q_0)\bar p_0)$. Then $u(s,t)=u(q(s,t),p(s,t))$, which is $\mathcal{C}^2$ in $(t,s)$.}. Let 
$$
\delta q(t)=\partial_sq(0,t)=\partial_{q_0}\varphi(t,q_0)\delta q_0
$$ 
and $\delta u=\partial_su(s,t)_{s=0}$. Since $(t,s)\mapsto q(s,t)$ is of class at least $\mathcal{C}^2$, we have
$$
\dot{\delta q}(t)=\partial_q(\xi_{q(t)}u(0,t))\delta q(t)+\xi_{q(t)}\delta u(t).
$$
Remarking that $p(0)(\delta q(0))=n(q_0)\bar p_0(\delta q_0)=0$, we get
$$\begin{aligned}
p(t_0)(\delta q(t_0))&=
\int_0^{t_0}\frac{d}{dt}\left( p(t)(\delta q(t))\right)dt\\
&= \int_0^{t_0}\left(-\partial_qh(q(t),p(t))\delta q(t)+p(t)\partial_q(\xi_{q(t)}u(t))\delta q(t)+p(t)\xi_{q(t)}\delta u(t)\right)dt.
\end{aligned}
$$
But $\partial_qh(q(t),p(t))=\partial_qH^1(q(t),p(t),u(t))=p(t)\partial_q(\xi_{q(t)}u(t))-\frac{1}{2}(\partial_qg_{q(t)})(u(t),u(t))$, and $p(t)\xi_{q(t)}\delta u(t)=g_{q(t)}(u(t),\delta u(t))$, so that
$$
\begin{aligned}
p(t_0)(\delta q(t_0))&=
\int_0^{t_0}\frac{d}{dt}\left( p(t)(\delta q(t))\right)dt\\
&= \int_0^{t_0}\left(\frac{1}{2}\partial_qg_{q(t)}(u(t),u(t))\delta q(t)+g_{q(t)}(u(t),\delta u(t))\right)dt\\
&= \int_0^{t_0}\partial_s\left(\underset{=h(q_0,n(q_0)\bar p_0)=c}{\underbrace{\frac{1}{2}g_{q(s,t)}(u(s,t),u(s,t))}}\right)_{s=0}dt\\
&=0.
\end{aligned}
$$
We get
$$
\varphi^*\theta=c^2dt.
$$
\end{proof}
Since $\varphi$ is a diffeomorphism, $\theta$ is indeed closed, which concludes the proof.
\end{proof}

\bibliographystyle{plain}
\bibliography{biblio}

\begin{thebibliography}{10}

\bibitem{ABCG}
A.~A. Agrachev, U.~Boscain, G.~Charlot, R.~Ghezzi, and M.~Sigalotti.
\newblock Two-dimensional almost-{R}iemannian structures with tangency points.
\newblock {\em Ann. Inst. H. Poincar\'e Anal. Non Lin\'eaire}, 27(3):793--807,
  2010.

\bibitem{AC}
A.~A. Agrachev and M.~Caponigro.
\newblock Controllability on the group of diffeomorphisms.
\newblock {\em Ann. Inst. H. Poincar\'e Anal. Non Lin\'eaire},
  26(6):2503--2509, 2009.

\bibitem{CTBOOK}
A.~A. Agrachev and Y.~L. Sachkov.
\newblock {\em Control theory from the geometric viewpoint}, volume~87 of {\em
  Encyclopaedia of Mathematical Sciences}.
\newblock Springer-Verlag, Berlin, 2004.
\newblock Control Theory and Optimization, II.

\bibitem{A2}
S.~Arguill\`ere.
\newblock The general setting for shape analysis.
\newblock {\em preprint}, 2015.

\bibitem{AMY}
S.~Arguill\`ere, M.~Miller, and L.~Youn{\`e}s.
\newblock L{D}{D}{M}{M} surface registration with atrophy constraints.
\newblock {\em preprint}, 2015.

\bibitem{AT}
S.~Arguill{\`e}re and E.~Tr{\'e}lat.
\newblock Sub-{R}iemannian structures on groups of diffeomorphisms.
\newblock {\em J. Inst. Math. Jussieu}, 2015.

\bibitem{ATY2}
S.~Arguill{\`e}re, E.~Tr{\'e}lat, A.~Trouv{\'e}, and L.~Youn{e}s.
\newblock Multiple shape registration using constrained optimal control.
\newblock {\em To appear in {S}{I}{A}{M} J. Imag. Sci.}, 2015.

\bibitem{ATY}
S.~Arguill{\`e}re, E.~Tr{\'e}lat, A.~Trouv{\'e}, and L.~Youn{e}s.
\newblock Shape deformation analysis from the optimal control viewpoint.
\newblock {\em J. Math. Pures Appl.}, 104(1):139--178, 2015.

\bibitem{SRBOOK}
A.~Bella{\"{\i}}che and J.-J. Risler, editors.
\newblock {\em Sub-{R}iemannian geometry}, volume 144 of {\em Progress in
  Mathematics}.
\newblock Birkh\"auser Verlag, Basel, 1996.

\bibitem{DS}
P.I. Dubnikov and S.N. Samborskii.
\newblock Controllability criterion for systems in a {B}anach space
  (generalization of chow's theorem).
\newblock {\em Ukrainian Mathematical Journal}, 32(5):429--432, 1980.

\bibitem{EM}
D.~G. Ebin and J.~Marsden.
\newblock Groups of diffeomorphisms and the motion of an incompressible fluid.
\newblock {\em Ann. of Math. (2)}, 92:102--163, 1970.

\bibitem{ES}
J.~Eichhorn and R.~Schmid.
\newblock Form preserving diffeomorphisms on open manifolds.
\newblock {\em Ann. Global Anal. Geom.}, 14(2):147--176, 1996.

\bibitem{GMV}
E.~Grong, I.~Markina, and A.~Vasil’ev.
\newblock Sub-riemannian geometry on infinite-dimensional manifolds.
\newblock {\em The Journal of Geometric Analysis}, 25(4):2474--2515.

\bibitem{KMBOOK}
A.~Kriegl and P.~W. Michor.
\newblock {\em The convenient setting of global analysis}, volume~53 of {\em
  Mathematical Surveys and Monographs}.
\newblock American Mathematical Society, Providence, RI, 1997.

\bibitem{LY}
X.~J. Li and J.~M. Yong.
\newblock {\em Optimal control theory for infinite-dimensional systems}.
\newblock Systems \& Control: Foundations \& Applications. Birkh\"auser Boston,
  Inc., Boston, MA, 1995.

\bibitem{MM}
P.~W. Michor and D.~Mumford.
\newblock An overview of the {R}iemannian metrics on spaces of curves using the
  {H}amiltonian approach.
\newblock {\em Appl. Comput. Harmon. Anal.}, 23(1):74--113, 2007.

\bibitem{MBOOK}
R.~Montgomery.
\newblock {\em A tour of subriemannian geometries, their geodesics and
  applications}, volume~91 of {\em Mathematical Surveys and Monographs}.
\newblock American Mathematical Society, Providence, RI, 2002.

\bibitem{O}
H.~Omori.
\newblock {\em Infinite dimensional {L}ie transformation groups}.
\newblock Lecture Notes in Mathematics, Vol. 427. Springer-Verlag, Berlin,
  1974.

\bibitem{PBOOK}
L.~S. Pontryagin, V.~G. Boltyanskii, R.~V. Gamkrelidze, and E.~F. Mishchenko.
\newblock {\em The mathematical theory of optimal processes}.
\newblock Translated by D. E. Brown. A Pergamon Press Book. The Macmillan Co.,
  New York, 1964.

\bibitem{S}
R.~Schmid.
\newblock Infinite dimensional {L}ie groups with applications to mathematical
  physics.
\newblock {\em J. Geom. Symmetry Phys.}, 1:54--120, 2004.

\bibitem{SUSS}
H.~J. Sussmann.
\newblock Orbits of families of vector fields and integrability of
  distributions.
\newblock {\em Trans. Amer. Math. Soc.}, 180:171--188, 1973.

\end{thebibliography}

\end{document}